\RequirePackage{fix-cm}

\documentclass[smallcondensed]{svjour3}      
\smartqed  
\usepackage{pgf,tikz}
\usepackage{mathrsfs}
\usepackage{amsmath}
\usepackage{amsfonts}
\usepackage{amssymb}

\usepackage{fancyhdr}
\usepackage{hyperref}
\usepackage{enumitem}
\usepackage{float}
\usetikzlibrary{arrows}
\usepackage[all]{xy}

\usepackage{color}

\usetikzlibrary{arrows}
\renewcommand{\figurename}{Figura}

\newcommand{\K}{\mathcal{K}}
\newcommand{\A}{\mathbf{A}}

\newcommand{\Can}[1]{\mathbf{#1}^{\sigma}} 
\newcommand{\MS}{\mathcal{MS}} 
\DeclareMathOperator{\Fi}{Fi}
\DeclareMathOperator{\Id}{Id}

\begin{document}

\title{A topological duality for monotone expansions of semilattices
\thanks{This research was supported by ANPCyT under grant 2019-00882. The third author has been funded by the European Research Council (ERC) under the European Union’s Horizon 2020 research and innovation program (grant agreement No. 670624).}}


\author{Ismael Calomino \and Paula Mench\'{o}n \and William J. Zuluaga Botero}


\institute{Ismael Calomino 
\at CIC and Departamento de Matem\'{a}tica, Facultad de Ciencias Exactas, Universidad Nacional del Centro, Tandil, Argentina \\
\email{calomino@exa.unicen.edu.ar}
\and Paula Mench\'{o}n 
\at CONICET and Departamento de Matem\'{a}tica, Facultad de Ciencias Exactas, Universidad Nacional del Centro, Tandil, Argentina \\
\email{mpmenchon@exa.unicen.edu.ar}
\and William J. Zuluaga Botero 
\at Laboratoire J. A. Dieudonn\'{e}, Universit\'{e} C\^ote d'Azur, Nice, France and Departamento de Matem\'{a}tica, Facultad de Ciencias Exactas, Universidad Nacional del Centro, Tandil, Argentina \\ \email{wizubo@gmail.com}}
                   
\date{Received: date / Accepted: date}
                    
\maketitle
\begin{abstract}
In this paper we provide a Stone style duality for monotone semilattices by using the topological duality developed in \cite{Celani2020} for semilattices together with a topological description of their canonical extension. As an application of this duality we obtain a characterization of the congruences of monotone semilattices by means of monotone lower-Vietoris-type topologies.
\end{abstract}

\section{Introduction}

In \cite{Birkhoff} Birkhoff proved that every finite distributive lattice is represented by the lower sets of a finite partial order. In \cite{Stone2} Stone generalized the results given by Birkhoff establishing a categorical duality for bounded distributive lattices through spectral spaces with continuous functions. Later, in \cite{Priestley I} Priestley proved another categorical duality for bounded distributive lattices by means of certain ordered topological spaces and continuous order-preserving maps. Along the years both Stone duality and Priestley duality have served as powerful tools not only for the study of bounded distributive lattices, but also as an inspiration for the study of many ordered algebraic structures associated with non-classical logics. Evidence of this is supplied by the several topological representations obtained for diverse classes of algebras with distributive semilattice reducts and semilattice reducts. This is the case of the Stone style duality for distributive semilattices developed in \cite{CelaniTopological,CelaniCalomino}, which extends Stone duality and completes the representation given by Gr\"{a}tzer in \cite{Gr78}, the Priestley style duality developed by Bezhanishvili and Jansana in \cite{Guram-Jansana I,Guram-Jansana III} for the class of distributive semilattices and implicative semilattices, and more recently the Stone style duality for semilattices and lattices obtained by Celani and Gonz\'{a}lez in \cite{Celani2020}.

A particular case of interest raises when noticing that any of these dualities can be combined with some other tools coming from algebra. This is the case of canonical extensions (\cite{Gehrke -Jonsson2000,GJ2,PalmigianoDunn,GPsemilattices}). The studies on this subject started with J\'{o}nsson and Tarski's works for Boolean algebras with operators (\cite{Jonsson y Tarski,JT2}). The most remarkable contribution of these papers was that they provided a procedure for transfer the benefits and the working methodology from the duality for Boolean algebras to several classes of algebras with additional operators. The crucial observation is that given a class of algebras with lattice reduct $\bf{L}$ and a suitable topological duality for $\bf{L}$, it is possible to build new topological dualities for monotone expansions of $\bf{L}$ by using a topological description of the canonical extension of the members of $\bf{L}$ in terms of the given duality. This approach was used by Celani and Mench\'{o}n in \cite{menchon2} to obtain a Stone style duality for monotone distributive semilattices. In \cite{MoshierJipsen} Moshier and Jipsen employed a duality different from the ones of Stone and Priestley to give a topological description of the canonical extension of a (not necessarily distributive) bounded lattice.

The aim of this paper is to establish a Stone style duality for monotone semilattices. To do so, we provide a topological description of the canonical extension of a semilattice by using the duality obtained in \cite{Celani2020}. We conclude by applying our duality to characterize the congruences of semilattices and monotone semilattices in terms of lower-Vietoris-type topologies (\cite{Ivanova-dimova}). These results generalize the results obtained in \cite{CelaniCalomino} for the class of distributive semilattices.

The paper is organized as follows. In Section \ref{preli} we give several notions and results which are needed for the rest of the paper. In Section \ref{section: Ideals} we present a topological description of the canonical extension of a semilattice in terms of the Stone style duality given in \cite{Celani2020} and we explore the underlying connection of our results with the description of canonical extensions of semilattices given by Gouveia and Priestley in \cite{GPsemilattices}. Further, we extend our results to a topological description of canonical extensions of monotone maps of semilattices. In Section \ref{duality} we establish our main result: a topological duality for the category of monotone semilattices and homomorphisms of monotone semilattices. Finally, in Section \ref{homomorphic images} as an application of the duality, we prove that the congruences of a semilattice and a monotone semilattice are in correspondence with the lower-Vietoris-type families and monotone lower-Vietoris-type topologies of their respective dual spaces. Moreover, we use such a characterization to obtain a version for semilattices of the results of Fajtlowicz and Schmidt in \cite{Fajtlowicz} concerning the lattice of congruences.

\section{Preliminaries} \label{preli}

If $f \colon X \rightarrow Y$ is a function and $U \subseteq X$ we write $f[U]$ for the direct image of $U$ through $f$. If $V \subseteq Y$ we write $f^{-1}[V]$ for the inverse image of $V$ through $f$. If $g \colon Y \rightarrow Z$ is a function we write $gf$ for the composition of $g$ with $f$. If $R \subseteq X \times Y$ is a binary relation and $x\in X$, then $R(x)=\{y\in Y\colon (x,y)\in R\}$. If $T \subseteq Y \times Z$ is some other binary relation, then the composition of $T$ with $R$ is the set $T \circ R = \{ (x,z) \colon \exists y \in Y [(x,y) \in R \ \text{and} \ (y,z) \in T ]\}$.

Let $\langle X,\leq \rangle$ be a poset. For each $Y \subseteq X$, let $[Y) = \{x \in X \colon \exists y \in Y (y \leq x)\}$ and $(Y] = \{x \in X \colon \exists y \in Y (x\leq y)\}$. We will say that $Y$ is an \emph{upset} of $X$ (a \emph{downset} of $X$) if $Y=[Y)$ ($Y=(Y]$). We write $\mathcal{P}(X)$ for the set of all subsets of $X$ and $\mathrm{Up}(X)$ for the set of all upsets of $X$. The complement of $Y \subseteq X$ will be denoted by $Y^{c}$. A subset $Z \subseteq X$ is called \emph{directed} if for every $x,y \in Z$, there exists $z \in Z$ such that $x \leq z$ and $y \leq z$. Analogously, $Z \subseteq X$ is called \emph{dually directed} if for every $x,y \in Z$, there exists $z \in Z$ such that $z \leq x$ and $z \leq y$. Two posets $\langle X, \leq_{X} \rangle$ and $\langle Y, \leq_{Y} \rangle$ are said to be \emph{dually isomorphic} if there exists a onto function $f \colon X \rightarrow Y$ such that $x \leq_{X} y$ if and only if $f(y) \leq_{Y} f(x)$. If the context is clear, the subscripts of the partial orders shall be omitted. 

A \emph{meet-semilattice with greatest element}, or \emph{semilattice} for short, is an algebra $\mathbf{A} = \langle A, \wedge, 1 \rangle$ of type $(2,0)$ such that the operation $\wedge$ is idempotent, commutative, associative, and $a\wedge1=a$ for all $a\in A$. The partial order $\leq$ on $\A$ is given by $a \leq b$ if and only if $a = a \wedge b$. For each poset $\langle X, \leq \rangle$, the structure $\langle \mathrm{Up}(X), \cap, X \rangle$ is a semilattice. If $\A$ is a semilattice, a subset $F \subseteq A$ is a \emph{filter} of $\A$ if it is an upset, $1\in F$ and if $a,b\in F$, then $a \wedge b\in F$. The set of all filters of $\A$ will be denoted by $\Fi({\A})$. The filter generated by the subset $X \subseteq A$ will be denoted by $F(X)$. If $X = \{a\}$, then $F(\{a\}) = [\{a\})$, or simply, $[a)$. We say that a proper filter $F\in \Fi({\A})$ is \emph{irreducible} if for all $F_{1},F_{2} \in \Fi(\A)$, if $F = F_{1} \cap F_{2}$, then $F = F_{1}$ or $F = F_{2}$. We write $X(\A)$ the set of all irreducible filters of $\A$.

\begin{theorem}(\cite{CelaniTopological})
Let $\A$ be a semilattice. A proper filter $F$ is irreducible if and only if for every $a,b \notin F$, there exist $c \notin F$ and $f \in F$ such that $a \wedge f \leq c$ and $b \wedge f \leq c$.
\end{theorem}

A nonempty subset $I \subseteq A$ is an \emph{order-ideal} of $\A$ if it is a directed downset. We denote by $\Id(\A)$ the set of all order-ideals of $\A$.

\begin{theorem} \label{separacion} (\cite{CelaniTopological}) Let $\A$ be a semilattice. Let $F \in \Fi(\A)$ and let $I \in \Id(\A)$. If $F \cap I = \emptyset$, then there exists $P \in X(\A)$ such that $F \subseteq P$ and $P \cap I = \emptyset$.
\end{theorem}

Let $\A$ and $\mathbf{B}$ be two semilattices. A map $h \colon A \rightarrow B$ is a \emph{homomorphism} if $h(1)=1$ and $h(a \wedge b) = h(a) \wedge h(b)$ for all $a,b \in A$. If we consider the poset $\langle X(\A),\subseteq \rangle$ and the map $\beta_{\A} \colon A \rightarrow \mathrm{Up}(X(\A))$ given by $\beta_{\A}(a) = \{P \in X(\A) \colon a\in P\}$, then it is proved that $\A$ is isomorphic to the subalgebra $\beta_{\A}[A] =\{ \beta_{\A}(a) \colon a\in A\}$ of $\langle \mathrm{Up}(X({\A})), \cap, X(\A) \rangle$. Throughout the paper and to simplify notation, we will omit the subscript of $\beta_{\A}$ where appropriate.

By last, if $\langle X, \tau\rangle$ is a topological space and $Y\subseteq X$, we write $\mathrm{cl}(Y)$ for the topological closure of $Y$. In particular, if $Y=\{y\}$, then $\mathrm{cl}(\{y\})=\mathrm{cl}(y)$. It is known that every topological space can be endowed with a partial order $\sqsubseteq$ defined by $x\sqsubseteq y$ if and only if $x\in \mathrm{cl}(y)$. Such an order is called the \emph{specialization order} of $X$.

\subsection{S-spaces}

The authors in \cite{Celani2020} developed topological dualities for semilattices and lattices that are natural generalizations of the Stone duality for bounded distributive lattices through spectral spaces. To do so, they proposed the following definition.

\begin{definition}  
Let $\A$ be a semilattice and let $F \in \Fi(\A)$. A subset $I \subseteq A$ is an \emph{$F$-ideal} if it is a downset and for every $a, b \in I$, there exist $c \in I$ and $f \in F$ such that $a \wedge f \leq c$ and $b \wedge f \leq c$.
\end{definition}

It is clear that a proper filter $F$ is irreducible if and only if $F^c$ is an $F$-ideal. The following result generalizes Theorem \ref{separacion}.

\begin{theorem} (\cite{Celani2020}) 
Let $\A$ be a semilattice. Let $F \in \Fi(\A)$ and let $I$ be an $F$-ideal. If $F \cap I = \emptyset$, then there exists $P \in X(\A)$ such that $F \subseteq P$ and $P \cap I = \emptyset$.
\end{theorem}

We stress that along this paper, by \emph{topological space} we mean pairs $\langle X,\mathcal{K}\rangle$ where $\langle X,\tau\rangle$ is a topological space and $\mathcal{K}$ is a subbase for $\tau$. We consider the following family of subsets of $X$:
\begin{equation*}
S(X) = \{U^c \colon U \in  \mathcal{K}\}.
\end{equation*}
Let $C_\mathcal{K}(X)$ be the closure system on $X$ generated by $S(X)$, i.e., $C_\mathcal{K}(X) = \{\bigcap A \colon A \subseteq S(X)\}$. The elements of $C_\mathcal{K}(X)$ are called \textit{subbasic closed subsets} of $X$. Note that $S(X) \subseteq C_\mathcal{K}(X) \subseteq C(X)$, where $C(X)$ is the collection of all closed subsets of the topology. A subset $Y \subseteq X$ is called \emph{saturated} if it is an intersection of open sets.

\begin{definition} 
Let $\langle X, \mathcal{K}\rangle$ be a topological space. Let $Y \subseteq X$. A family $\mathcal{Z} \subseteq S(X)$ is a \emph{$Y$-family} if for every $A, B \in \mathcal{Z}$, there exist $H, C \in S(X)$ such that $Y \subseteq H$, $C \in \mathcal{Z}$, $A \cap H \subseteq C$ and $B \cap H \subseteq C$.
\end{definition}

\begin{definition} (\cite{Celani2020}) \label{S-spaces}
An \emph{$S$-space} is a topological space $\langle X, \mathcal{K} \rangle$ satisfying the following conditions:
\begin{enumerate} [label=(S{\arabic*})]
\item $\langle X, \mathcal{K} \rangle$ is a $T_0$-space and $X = \bigcup \mathcal{K}$.
\item $\mathcal{K}$ is a subbase of compact open subsets, it is closed under finite unions and $\emptyset\in \mathcal{K}$.
\item \label{S3} For every $U, V\in \mathcal{K}$, if $x \in U \cap V$, then there exist $W, D \in\mathcal{K}$ such that $x \notin W$, $x \in D$ and $D \subseteq (U \cap V) \cup W$.
\item \label{S4} If $Y \in C_\mathcal{K}(X)$ and $\mathcal{J} \subseteq S(X)$ is a $Y$-family such that $Y \cap A^c \neq \emptyset$, for all $A \in \mathcal{J}$, then $Y \cap \bigcap \{A^c \colon A \in \mathcal{J}\} \neq \emptyset$.
\end{enumerate}
\end{definition}

Let $\langle X, \K \rangle$ be an $S$-space. Then the structure $\mathbf{S}(X) = \langle S(X), \cap, X \rangle$ is a semilattice, called the \emph{dual semilattice of $\langle X,\mathcal{K}\rangle$}. Conversely, if $\A$ is a semilattice, then $\langle X(\A), \K_\A \rangle$ is a topological space where $\K_\A = \{\beta(a)^c \colon a \in A\}$ is a subbase. It follows that $\langle X(\A), \K_\A \rangle$ is an $S$-space, called the \emph{dual $S$-space of $\A$}, and $\beta \colon A \rightarrow S(X(\A))$ is an isomorphism between semilattices. If we consider the associated $S$-space $\langle X(\mathbf{S}(X)), \K_{\mathbf{S}(X)} \rangle$ of the dual semilattice $\mathbf{S}(X)$, then the mapping $H_X \colon X \rightarrow X(\mathbf{S}(X))$ given by $H_X(x) = \{A \in S(X) \colon x \in A\}$ for all $x \in X$ is a homeomorphism between $S$-spaces.

Now, we mention the relationship existing between the subbasic closed subsets and the filters of the dual semilattice of certain topological spaces proved in \cite{Celani2020}. 

\begin{proposition} \label{prop_subclos and filters}
Let $\langle X,\mathcal{K} \rangle$ be a topological space such that $\mathcal{K}$ is a subbase of compact open subsets, and it is closed under finite unions and $\emptyset \in \mathcal{K}$. Then:
\begin{enumerate}
\item For every $Y \in C_\mathcal{K}(X)$, $F_Y = \{A \in S(X) \colon Y \subseteq A\} \in \Fi(\mathbf{S}(X))$.
\item For every $F \in \Fi(\mathbf{S}(X))$, $Y_F = \bigcap F \in C_\mathcal{K}(X)$. 
\item $Y = Y_{F_Y}$ for all $Y \in C_\mathcal{K}(X)$ and $F = F_{Y_F}$ for all $F \in \Fi(\mathbf{S}(X))$.
\end{enumerate}
Moreover, the posets $\langle C_\mathcal{K}(X), \subseteq \rangle$ and $\langle \Fi(\mathbf{S}(X)), \subseteq \rangle$ are dually isomorphic.
\end{proposition}

We conclude this part by making explicit Proposition \ref{prop_subclos and filters} in terms of semilattices and their dual S-spaces. Such a restatement will be extensively used along the rest of the paper.

Let $\A$ be a semilattice. Then the mapping $\varphi \colon \Fi(\A) \rightarrow C_{\mathcal{K}_{\mathbf{A}}}(X(\A))$ defined by
\begin{equation}\label{definition of fi}
\varphi(F) = \bigcap\{\beta(a) \colon a\in F\} = \{P \in X(\A) \colon F \subseteq P\}
\end{equation} 
is a dual isomorphism, whose inverse $\psi \colon C_{\mathcal{K}_{\A}}(X(\A)) \rightarrow \Fi(\A)$ is given by 
\begin{equation}\label{definition of psi}
\psi(Y) = \{a \in A \colon Y\subseteq \beta(a)\}.
\end{equation}

\subsection{Canonical extension}

Inspired in \cite{PalmigianoDunn}, we present the basic definitions of the theory of canonical extensions of semilattices. These definitions also coincide with the definition of canonical extension when restricted to bounded distributive lattices and Boolean algebras (\cite{Jonsson y Tarski,JT2,Gehrke -Jonsson2000,GJ2}).

\begin{definition}
Let $\mathbf{A}$ be a semilattice. A \emph{completion} of $\A$ is a pair $\mathbf{E} = \langle E, e \rangle$ where $E$ is a complete lattice and the map $e \colon A \rightarrow E$ is an embedding. 
\end{definition}

Let $\mathbf{A}$ be a semilattice and let $\mathbf{E}$ be a completion of $\A$. An element $x \in E$ is called \emph{closed} if there is $F \in \Fi(\A)$ such that $x = \bigwedge e(F)$. An element $x \in E$ is called \emph{open} if there is $I \in \Id(\A)$ such that $x = \bigvee e(F)$. Let us denote the collection of all closed elements of $E$ by $K(\mathbf{E})$ and the collection of all open elements of $E$ by $O(\mathbf{E})$. The completion $\mathbf{E}$ is \emph{dense} if $x = \bigwedge \{ y \in O(\mathbf{E}) \colon x \leq y \}$ and $x = \bigvee \{ y \in K(\mathbf{E}) \colon y \leq x \}$ for all $x \in E$. The completion $\mathbf{E}$ is \emph{compact} if for every non-empty dually directed subset $D$ of $A$ and every non-empty directed subset $U$ of $A$ such that $\bigwedge e(D) \leq \bigvee e(U)$, then there exist $x \in D$ and $y \in U$ such that $x\leq y$.

From now on we will suppress the embedding $e$ and we assume that $\A$ is a subalgebra of a completion $\mathbf{E}$ of $\A$.

\begin{definition}
Let $\A$ be a semilattice. A completion $\mathbf{E}$ of $\A$ is a \emph{canonical extension} if it is dense and compact.
\end{definition}

\begin{theorem}
Let $\A$ be a semilattice. Then $\A$ has a canonical extension and it is unique up to isomorphism.
\end{theorem}

\section{Canonical extensions of semilattices} \label{section: Ideals}

The aim of this section is to exhibit a topological description of the canonical extension of a semilattice in terms of the duality obtained in \cite{Celani2020}. As a result we obtain that the canonical extension of a semilattice is determined by a closure system generated by the set of subasic saturated subsets of its dual space. Our motivations differ from those of \cite{GPsemilattices} because we are interested in producing a duality for monontone semilattices. Further details of the relation between our construction and the Gouveia and Priestley construction will be provided in Section \ref{Connection with the construction of Maidana y Lopez}.

\begin{definition} 
Let $\langle X,\mathcal{K} \rangle$ be an $S$-space. A subset $Z \subseteq X$ is a \emph{subbasic saturated subset} if there exists a dually directed family $\mathcal{L} \subseteq \mathcal{K}$ such that $Z = \bigcap \{ W \colon W \in \mathcal{L} \}$.
\end{definition}

We denote by $\mathcal{Z}(X)$ the set of all subbasic saturated subsets of an $S$-space $\langle X, \mathcal{K} \rangle$. Note that every subbasic saturated subset is a saturated set.

\begin{lemma} \label{lema Y-familia}
Let $\langle X, \K \rangle$ be an $S$-space. Then: 
\begin{enumerate}
\item If $\mathcal{L} \subseteq \K$ is a dually directed family, then $\mathcal{A}_{\mathcal{L}} = \{U^c \colon U \in \mathcal{L} \} \subseteq S(X)$ is a $Y$-family for all $Y \in C_\K(X)$.
\item If $\mathcal{A} \subseteq S(X)$ is a $Y$-family for all $Y\in C_\K(X)$, then $\mathcal{L}_{\mathcal{A}} = \{U^c \colon U \in \mathcal{A} \}\subseteq \K$ is a dually directed family.
\end{enumerate}
Moreover, the assignments $\mathcal{L} \mapsto \mathcal{A}_{\mathcal{L}}$ and $\mathcal{A}\mapsto \mathcal{L}_{\mathcal{A}}$ establish a bijection between the set of dually directed families of $\mathcal{K}$ and the families of $C_\K(X)$ which are a $Y$-family for all $Y \in C_\K(X)$.
\end{lemma}

\begin{proof}
1. Let $Y \in C_\K(X)$ and $U,V \in \mathcal{L}$. Since $\mathcal{L}$ is a dually directed family, there exists $W \in \mathcal{L}$ such that $W \subseteq U \cap V$. From $X \in S(X)$, we get $Y \subseteq X$, $U^c \cap X = U^c \subseteq W^c$ and $V^c \cap X = V^c \subseteq W^c$. So, $\mathcal{A}_{\mathcal{L}} = \{U^c \colon U \in \mathcal{L}\}$ is a $Y$-family.

2. Let $\mathcal{A} \subseteq S(X)$ be a $Y$-family for all $Y \in C_\K(X)$. Let $U,V \in \mathcal{A}$. As $X \in C_\K(X)$, there exists $W \in \mathcal{A}$ such that $U \cap X \subseteq W$ and $V \cap X \subseteq W$. Therefore, $\mathcal{L}_{\mathcal{A}} = \{U^c \colon U \in \mathcal{A} \} \subseteq \K$ is a dually directed family.

For the moreover part, observe that $\mathcal{L}_{\mathcal{A}_{\mathcal{L}}}=\mathcal{L}$ and $\mathcal{A}_{\mathcal{L}_{\mathcal{A}}}=\mathcal{A}$.
\end{proof}

\begin{remark} \label{todo saturado es Y familia}
If $Z = \bigcap \{W \colon W \in \mathcal{L} \} \in \mathcal{Z}(X)$, then the collection $\{W^c \colon W \in \mathcal{L} \}$ is a $Y$-family for all $Y \in C_\K(X)$.
\end{remark}

\begin{proposition}
Let $\langle X, \mathcal{K} \rangle$ be an $S$-space. The subbasic saturated subsets are compact saturated subsets of the topology. 
\end{proposition}

\begin{proof}
Let $Z \in \mathcal{Z}(X)$. We prove that $Z$ is compact. By the Alexander's subbasis Theorem, we only need to prove that for every cover $\{O_i \colon i \in I\}$ of $Z$ by elements of $\K$, there is a finite subcover. By assumption, there exists a dually directed family $\mathcal{L} \subseteq \mathcal{K}$ such that $Z = \bigcap \{W \colon W \in \mathcal{L}\}$. Since $Z \subseteq  \bigcup \{ O_i \colon i \in I \}$, we have  
\begin{equation*}
\bigcap \{W \colon W \in \mathcal{L} \} \cap \bigcap \{ O_i^c \colon i \in I \}=\emptyset.
\end{equation*}
By \ref{S4} and Lemma \ref{lema Y-familia} there exists $W \in \mathcal{L}$ such that $W \cap \bigcap \{ O_i^c \colon i \in I \} = \emptyset$. Then $W \subseteq \bigcup \{ O_i \colon i \in I \}$ and as $W$ is a compact subset, there exist $O_{i_1}, \dots, O_{i_n}$ such that $W \subseteq O_{i_1} \cup \ldots \cup O_{i_n}$. Hence, $Z \subseteq W \subseteq O_{i_1} \cup \ldots\cup O_{i_n}$ and $Z$ is compact.
\end{proof}

Let $\mathbf{A}$ be a semilattice and let $I\in$ $\mathrm{Id}({\mathbf{A}})$. We consider the following subset
\begin{equation*}
\alpha(I) = \bigcap \{\beta(a)^{c} \colon a \in I \} = \{ P \in X({\A}) \colon I \cap P = \emptyset\}.
\end{equation*}
Then $\alpha(I)$ is a subbasic saturated subset of the dual $S$-space $\langle X({\mathbf{A}}), \mathcal{K}_{\mathbf{A}} \rangle$.
Conversely, let $Z\subseteq X({\mathbf{A}})$ be a subbasic saturated subset and consider 
\begin{equation*}
I_{\mathbf{A}}(Z) = \{a \in A \colon \beta(a)\cap Z = \emptyset\}.
\end{equation*}
It is easy to see that $I_{\mathbf{A}}(Z)$ is a downset of $\mathbf{A}$.

\begin{proposition} 
Let $\langle X, \K \rangle$ be an $S$-space and let $Z \subseteq X$ be a saturated set such that $Z = \bigcap \{W \colon W \in \mathcal{J} \}$ for some family $\mathcal{J} \subseteq \K$. Then $Z$ is a subbasic saturated subset if and only if the family $\{ U \in \K \colon Z \subseteq U\}$ is dually directed.
\end{proposition}

\begin{proof}
If $Z$ is a subbasic saturated subset, then there exists a dually directed family $\mathcal{L} \subseteq \K$ such that $Z = \bigcap \{W \colon W \in \mathcal{L} \}$. Let $S,T \in \{U \in \K \colon Z \subseteq U\}$. Thus, $Z \cap S^c = \emptyset$ and $Z \cap T^c = \emptyset$. By \ref{S4}, there exist $S',T' \in \mathcal{L}$ such that $S' \cap S^c = \emptyset$ and $T' \cap T^c = \emptyset$. As $\mathcal{L}$ is dually directed, there is $W \in \mathcal{L}$ such that $W \subseteq S'$ and $W \subseteq T'$, i.e., $W \subseteq S$ and $W \subseteq T$. So, $\{U \in \K \colon Z \subseteq U\}$ is dually directed.

Conversely, suppose that the family $\{U \in \K \colon Z \subseteq U\}$ is dually directed. Since $Z = \bigcap \{W \colon W \in \mathcal{J} \}$ for some family $\mathcal{J} \subseteq \K$, then $Z \subseteq \bigcap \{U \in \K \colon Z \subseteq U\} \subseteq \bigcap \{W \colon W \in \mathcal{J} \}=Z$ and $Z$ is a subbasic saturated subset.
\end{proof}

In the following result we see that order-ideals are in one to one correspondence with the subbasic saturated subsets.

\begin{theorem} \label{Dualidad Ideales}
Let $\A$ be a semilattice. Then the posets $\langle \mathrm{Id}(\A), \subseteq \rangle$ and $\langle \mathcal{Z}(X(\A)), \subseteq \rangle$ are dually isomorphic.
\end{theorem}

\begin{proof}
Let $Z \in \mathcal{Z}(X(\A))$. We see that $I_{\A}(Z)$
is an order-ideal of $\A$ and $Z = \alpha(I_{\A}(Z))$. Moreover, if $I \in {\rm{Id}}(\A)$, then we prove that $I = I_{\A}(\alpha(I))$.

It is clear that $I_{\A}(Z)$ is a downset of $\A$. Let $a,b \in I_{\A}(Z)$. So, $\beta(a) \cap Z = \emptyset$ and $\beta(b) \cap Z = \emptyset$. Since $Z = \bigcap \{ \beta(x)^{c} \colon \beta(x)^{c} \in \mathcal{L}\}$ for some dually directed family $\mathcal{L} \subseteq \mathcal{K}_{\A}$ and $\beta(a),\beta(b) \in C_\K(X(\A))$, then from Remark \ref{todo saturado es Y familia} $\{ \beta(x) \colon \beta(x)^{c} \in \mathcal{L}\}$ is a $\beta(a)$-family and a $\beta(b)$-family. Whence, by \ref{S4} there exist $\beta(c)^{c},\beta(d)^c \in \mathcal{L}$ such that $\beta(c)^{c} \cap \beta(a) = \emptyset$ and $\beta(d)^{c} \cap \beta(b) = \emptyset$. Thus, there is $\beta(e)^c \in \mathcal{L}$ such that $\beta(e)^c \subseteq \beta(c)^c, \beta(d)^c$. So $\beta(e)^{c} \cap \beta(a) = \emptyset$, $\beta(e)^{c} \cap \beta(b) = \emptyset$ and $Z \cap \beta(e) = \emptyset$, i.e., $a,b \leq e$ and $e \in I_{\A}(Z)$. Therefore, $I_{\mathbf{A}}(Z)$ is an order-ideal of $\A$ and 
\begin{equation*}
\alpha(I_{\A}(Z)) = \bigcap\{ \beta(a)^{c} \colon Z \subseteq \beta(a)^{c} \} \subseteq \bigcap \{\beta(x)^{c} \colon \beta(x)^{c} \in \mathcal{L} \} = Z.
\end{equation*}
The other inclusion is immediate.

Let $I \in {\rm{Id}}(\A)$. If $b \in I_{\A}(\alpha(I))$, then $\beta(b) \cap \alpha(I) = \beta(b) \cap \bigcap \{\beta(a)^{c} \colon a \in I \} = \emptyset$. Since $\beta(b) \in C_\K(X(\A))$ and the family $\{ \beta(a)^{c} \colon a \in I \}$ is dually directed, by Remark \ref{todo saturado es Y familia} and \ref{S4} we get that there is $c \in I$ such that $\beta(b) \cap \beta(c)^{c} = \emptyset$, i.e., $\beta(b) \subseteq \beta(c)$. So, $b \leq c$ and as $I$ is a downset, we have $b \in I$. The other inclusion is straightforward.

Then we have a mapping $\alpha \colon \mathrm{Id}(\A) \rightarrow \mathcal{Z}(X(\A))$ with inverse $I_{\A} \colon \mathcal{Z}(X(\A))\rightarrow\mathrm{Id}(\A)$. Let $I_{1}, I_{2} \in \mathrm{Id}(\A)$. If $I_{1}\subseteq I_{2}$ and $P \in \alpha(I_{2})$, then $P \cap I_{2} = \emptyset$. It follows
that $P \cap I_{1} = \emptyset$ and $P \in \alpha(I_{1})$. Reciprocally, assume that $\alpha(I_{1}) \subseteq \alpha(I_{2})$. Let $a \in I_{2}$ and suppose $a \notin I_{1}$. By Theorem \ref{separacion} there exists $P \in X(\A)$ such that $a \in P$ and $P \cap I_{1} = \emptyset$. So, $P \in \alpha(I_{1})$ and $P \notin \alpha(I_{2})$ which is a contradiction. Then $a \in I_{1}$ and $I_{2} \subseteq I_{1}$. Therefore, $\langle \mathrm{Id}(\A), \subseteq \rangle$ and $\langle \mathcal{Z}(X(\A)), \subseteq \rangle$ are dually isomorphic.
\end{proof}

\begin{lemma}\label{Compacidad}
Let $\A$ be a semilattice. If $Y \in C_{\K_{\A}}(X(\A))$ and $Z \in \mathcal{Z}(X(\A))$, then $\psi(Y) \cap I_{\A}(Z) \neq \emptyset$ if and only if $Y \cap Z= \emptyset$.
\end{lemma}

\begin{proof}
Suppose $Y \cap Z=\emptyset$. Since $Z = \bigcap \{\beta(a)^{c} \colon \beta(a)^{c} \in \mathcal{L}\}$ for some dually directed family $\mathcal{L} \subseteq \mathcal{K}_{\A}$, by Remark \ref{todo saturado es Y familia} and \ref{S4} there is $\beta(c)^{c} \in \mathcal{L}$ such that $Y \cap \beta(c)^{c} = \emptyset$. So, there exist $c \in A$ such that $c \in \psi(Y)$ and $c \in I_{\mathbf{A}}(Z)$. Hence, $\psi(Y) \cap I_{\mathbf{A}}(Z) \neq \emptyset$. The other implication is straightforward.
\end{proof}

\begin{remark}\label{Compacidad Filtros Ideales}
Observe that Lemma \ref{Compacidad} can be rephrased in terms of order-ideals and filters of $\mathbf{A}$ as follows: Let $\A$ be a semilattice. If $F \in \mathrm{Fi}(\A)$ and $I \in \mathrm{Id}(\A)$, then $\varphi(F) \cap \alpha(I) = \emptyset$ if and only if $F \cap I \neq \emptyset$.
\end{remark}

Let $\A$ be a semilattice and let us consider the following closure operator $\Lambda \colon \mathrm{Up}(X({\mathbf{A}})) \rightarrow \mathrm{Up}(X({\mathbf{A}}))$ given by 
\begin{equation}\label{operador lambda}
\Lambda(Y) = \bigcap \{ U^{c} \colon U \in \mathcal{Z}(X({\mathbf{A}})) \text{ and } Y \subseteq U^{c} \}.
\end{equation}

We call $E(X(\A))$ to the closure system on $\mathrm{Up}(X(\A))$ generated by $\mathcal{Z}(X({\mathbf{A}}))$: 
\begin{equation*}
E(X(\A)) = \left\{\bigcap \{U^{c} \colon U \in \mathcal{B}\} \colon \mathcal{B} \subseteq  \mathcal{Z}(X(\A)) \right\}.
\end{equation*} 

Notice that the structure $\langle E(X(\A)), \vee, \cap, \emptyset, X(\A) \rangle$ is a complete lattice, where the join of an arbitrary family $\{C_i \colon i \in I \}$ of $E(X(\A))$ is given by 
\begin{equation*}
\bigvee \{ C_i \colon i \in I \} = \bigcap \left\{ G \in E(X(\A)) \colon \bigcup \{ C_i \colon i \in I \} \subseteq G \right\} = \Lambda \left(\bigcup \{ C_i \colon i \in I \} \right).
\end{equation*}

\begin{lemma} \label{lema supremo} 
Let $\A$ be a semilattice. Then:
\begin{enumerate}
\item $\beta (a) \in E(X(\A))$ for all $a \in A$ and $C_\K(X(\A))\subseteq E(X(\A))$.

\item If $a,b \in A$ such that $a \vee b$ exists, then $\beta( a \vee b) = \beta(a) \vee \beta(b)$. 

\item Let $D\subseteq A$ be a directed subset. Then 
\begin{equation*}
 \bigvee \{\beta(a) \colon a \in D \} = \bigcup \{\beta(a) \colon a\in D \}.
\end{equation*}
In particular, if $U \in \mathcal{Z}(X(\A))$, we have 
\begin{equation*}
U^c = \bigcup\{\beta(a) \colon a \in I_\A(U)\} = \bigvee \{\beta(a) \colon a\in I_\A(U)\}.
\end{equation*} 
\end{enumerate}
\end{lemma}

\begin{proof}
1. It follows from the fact that $\beta(a)^c \in \mathcal{Z}(X(\A))$ for all $a \in A$.

2. Let $a,b \in A$ be such that $a \vee b$ exists. It is clear that $\beta(a) \vee \beta (b) \subseteq \beta (a \vee b)$. On the other hand, let $P \in \beta (a\vee b)$ and $U \in \mathcal{Z}(X(\A))$ be such that $\beta(a) \cup \beta(b) \subseteq U^c$. Then $\beta(a) \cap U = \emptyset$ and $\beta(b) \cap U = \emptyset$. So, $a,b \in I_\A(U)$ and $a \vee b \in I_\A(U)$. Thus, $P \cap I_\A(U) \neq \emptyset$ and we get $P \notin U$, i.e, $P\in U^c$. Therefore, $P \in \beta(a) \vee \beta(b)$.

3. Let $D \subseteq A$ be a directed subset. Then $\bigcup \{\beta(a) \colon a \in D\} \subseteq \bigvee\{\beta(a) \colon a \in D\}$. Let $P \in \bigvee \{\beta(a) \colon a \in D\}$ and consider the set $(D]$. Then we have $(D] \in \mathrm{Id}(\A)$. Since $\beta(a) \subseteq \alpha((D])^c$ for all $a \in D$, then $\bigcup \{\beta(a) \colon a\in D \} \subseteq \alpha((D])^c$. Thus, $P \notin \alpha((D])$ and $P \cap (D] \neq \emptyset$. It follows that there is $a \in D$ such that $a \in P$ which implies $P \in \bigcup \{\beta(a) \colon a \in D\}$.
\end{proof}

Now we are ready to prove the main result of this section.

\begin{theorem} \label{canonical extension of semilattices}
Let $\A$ be a semilattice. Then the pair $\langle E(X(\A)), \beta \rangle$ is a canonical extension of $\A$. 
\end{theorem}

\begin{proof}
From Lemma \ref{lema supremo}, $\beta(a) \in E(X(\A))$ for all $a \in A$ and $\beta$ is an embedding. 

In order to prove that $E(X(\A))$ is dense, let $V \in E(X(\A))$. Then there exists $\mathcal{B} \subseteq \mathcal{Z}(X(\A))$ such that $V = \bigcap \{ U^{c} \colon U \in \mathcal{B} \}$. From Theorem \ref{Dualidad Ideales}, it follows that 
\begin{equation*}
U^{c} = \left( \bigcap \{ \beta(a)^{c} \colon a \in I_{\A}(U) \} \right)^{c} = \bigcup \beta[I_{\A}(U)] = \bigvee \beta[I_{\A}(U)]
\end{equation*} 
and
\begin{equation*}
V = \bigcap \left\{ \bigvee \beta[I_{\A}(U)] \colon U \in \mathcal{B} \right\}.
\end{equation*}
On the other hand, since $V \in \mathrm{Up}(X(\A))$, it follows that $V = \bigcup \{ \varphi(P) \colon P \in V \}$. It is easy to check that
\begin{equation*}
V = \bigcup \{ \varphi(P) \colon P \in V \} = \bigvee \{ \varphi(P) \colon P \in V \} = \bigvee \left\{ \bigcap \{ \beta(P) \colon P \in V \} \right\}.
\end{equation*} 
 
Finally, we prove that $E(X(\A))$ is compact. Let $F \in \mathrm{Fi}(\A)$ and $I \in \mathrm{Id}(\A)$ be such that $\bigcap \beta[F] \subseteq \bigcup \beta[I]$. So, $\bigcap \beta[F] \cap \left( \bigcup \beta[I] \right)^{c} = \emptyset$ and by Theorem \ref{Compacidad} and Remark \ref{Compacidad Filtros Ideales}, we have $F \cap I \neq \emptyset$. Hence, there exists $a \in A$ such that $\beta(a) \in \beta[F]$ and $\beta(a) \in \beta[I]$. The result follows. 
\end{proof}

As usual, we write $\Can{A}$ to denote the canonical extension of a semilattice $\mathbf{A}$.

\begin{remark} \label{Abiertos Cerrados Extension canonica}
Let $\A$ be a semilattice and $\langle X(\A),\K_{\A} \rangle$ its dual $S$-space. As an immediate consequence of Theorem \ref{Dualidad Ideales} and Proposition 3.2 of \cite{Celani2020}, we have 
\begin{equation*}
K(\Can{A}) = C_{\mathcal{\mathcal{K}_{\mathbf{A}}}}(X(\mathbf{A})) {\hspace{0.3cm}} \text{and} {\hspace{0.3cm}} O(\Can{A})=\{Z^{c}\colon Z\in \mathcal{Z}(X(\mathbf{A}))\},
\end{equation*}
i.e., the closed elements of the canonical extension of $\A$ are exactly the subbasic closed subsets of $\langle X(\A),\K_{\A} \rangle$ and the open elements of the canonical extension of $\A$ are exactly the complements of the subbasic saturated subsets of $\langle X(\A),\K_{\A} \rangle$.
\end{remark}

\subsection{Connection with the construction of Gouveia and Priestley} \label{Connection with the construction of Maidana y Lopez}
 
Canonical extensions of semilattices have been studied in detail by Gouveia and Priestley in \cite{GPsemilattices}. This work is connected with the topological construction of canonical extensions made by Moshier and Jipsen \cite{MoshierJipsen}. Their approach places strong emphasis on the iterated filter completion and the canonical extension: they consider a semilattice $\A$ and the lattice of filters of filters of $\A$, denoted by $\Fi^2(\A)=\Fi(\Fi(\A))$. The family of sets $\Fi^2(\A)$ is an algebraic closure system and forms an algebraic lattice in which arbitrary meet and directed join are given, respectively, by intersection and directed union. The join of an arbitrary family of elements $\{F_i \colon i \in I \}$ is given by 
\begin{equation*}
\displaystyle \bigcap \left\{ \mathcal{G} \in \Fi^2(\A) \colon \bigcup \{F_i \colon i \in I \} \subseteq \mathcal{G} \right\}. 
\end{equation*}
In \cite{GPsemilattices}, the authors define an order embedding $e \colon A \rightarrow \Fi^2(\A)$ given by  $e(a) = \{ F  \in \Fi(\A) \colon a \in F\}$ for all $a \in A$ and the pair $\langle \Fi^2(\A), e \rangle$ is a completion of $\A$. Since $\Fi^2(\A)$ may not  be a canonical extension of $\A$, to construct a canonical extension the authors consider the subset of $\Fi^2(\A)$ consisting of those of its elements which are meets of directed joins from $e(A)$, i.e.,
\begin{equation*}
C=\left\{\bigcap \bigcup e(I_j) \colon \{I_j \colon j \in J \} \subseteq \Id(\A) \right\}.
\end{equation*}

Since the canonical extension is unique up to an isomorphism that fixes $\A$, there is an isomorphism between the canonical extension given in \cite{GPsemilattices} and the one we have developed in this work from the topological duality introduced in \cite{Celani2020}. We now proceed to the details. We know by Theorem \ref{canonical extension of semilattices} that $\langle E(X(\A)), \beta \rangle$ is a canonical extension for $\A$, where
\begin{equation*}
E(X(\A)) = \left\{ \bigcap \{U^{c} \colon U \in \mathcal{B}\} \colon \mathcal{B} \subseteq \mathcal{Z}(X(\mathbf{A})) \right\}
\end{equation*}
and $\beta \colon A \rightarrow E(X(\A))$ is given by $\beta (a) = \{P \in X(\A) \colon a\in P\}$. So, if we define the mapping $i \colon C \rightarrow E(X(\A))$ by
\begin{equation*}
i(X) = \{ P \in X(\A) \colon P \in X\} = X \cap X(\A)
\end{equation*}
then it is easy to prove that $i(X) = \bigcap \{ \alpha (I)^c \colon I \in \Id(\A) \text{ and } X \subseteq \bigcup e(I)\}$ and $i(e(a))=\beta(a)$. We get that $i$ is an isomorphism of lattices that fixes $\A$ with inverse $j \colon E(X(\A)) \rightarrow C$ given by
\begin{equation*}
j(Y) = \bigcap \left\{\bigcup e(I) \colon I \in \Id(\A) \text{ and } Y \subseteq \alpha(I)^c \right\}.
\end{equation*}
It is also worth noting that $i(\bigcup e(I)) = \alpha(I)^c$ for all $I \in \Id(\A)$ and $i(\bigcap e(F)) = \varphi (F) = Y_F$ for all $F \in \Fi(\A)$.

\subsection{Extension of monotone maps} \label{Extension of monotonic maps}

In this section we study canonical extensions of monotone maps of semilattices by using Theorem \ref{canonical extension of semilattices}. We start by recalling some results about canonical extensions of posets. For further details the reader may consult \cite{PalmigianoDunn}.

Let $P$ and $Q$ be two posets and $f \colon P \rightarrow Q$ be an order preserving map. Let $P^\sigma$ and $Q^\sigma$ be the canonical extensions of $P$ and $Q$, respectively. We define the maps $f^\sigma, f^\pi \colon P^\sigma \to Q^\sigma$ by 
\begin{align}
f^\sigma(u) & = \bigvee \left\{ \bigwedge \{ f(p) \colon x\leq p \in P\} \colon u \geq x \in K(P^\sigma) \right \}, \label{ext1} \\
f^\pi(u)& = \bigwedge \left \{\bigvee\{ f(p) \colon  y \geq p\in P\} \colon u\leq y \in O(P^\sigma) \right \}. \label{ext2}
\end{align}
For every order preserving map $f \colon P \to Q$, both $f^\sigma$ and $f^\pi$ are order preserving extensions of $f$. In addition, $f^\sigma \leq f^\pi$ with equality holding on both $K (P^\sigma)$ and $O(P^\sigma)$. For every $u \in P^\sigma$, $x \in K(P^\sigma)$ and $y \in O(P^\sigma)$ we have
\begin{align*}
f^\sigma(x) & = \bigwedge \{f(p) \colon x \leq p \in P\}, \\
f^\sigma(u) & = \bigvee \{f^\sigma (x) \colon u \geq x\in K(P^\sigma)\}, \\
f^\pi(y) & =  \bigvee \{f(p) \colon y \geq p\in P \}, \\
f^\pi (u) & = \bigwedge \{f^\pi(y) \colon u \leq y \in O(P^\sigma)\}.
\end{align*}
Also, $f^\sigma$ and $f^\pi$ map closed elements to closed elements and open elements to open elements. In general the extensions $f^\sigma$ and $f^\pi$ do not coincide. 

Since every semilattice is a poset, we can borrow the properties described before to this context. The following result is a straightforward consequence of Lemma 3.4 of \cite{PalmigianoDunn}, Theorem \ref{canonical extension of semilattices} and Remark \ref{Abiertos Cerrados Extension canonica}.

\begin{lemma} \label{Topological extensions of maps}
Let $\mathbf{A}$ and $\mathbf{B}$ be two semilattices and $f \colon A \rightarrow B$ be an order preserving map. Consider the mappings $f^{\sigma}, f^{\pi} \colon \Can{A} \rightarrow \Can{B}$ defined in (\ref{ext1}) and (\ref{ext2}). Then for every $a \in A$, $f^{\sigma}(\beta_{\mathbf{A}}(a)) = f^{\pi}(\beta_{\mathbf{A}}(a)) = \beta_{\mathbf{B}}(f(a))$. Moreover, $f^{\sigma} \leq f^{\pi}$ with the equality holding in $K(\Can{A}) \cup O(\Can{A})$. For every $Y \in C_{\mathcal{K}_{\mathbf{A}}}(X(\mathbf{A}))$, $Z \in \mathcal{Z}(X(\mathbf{A}))$ and $V \in E(X(\A))$ the following hold:
\begin{enumerate}
\item $\displaystyle f^\sigma(Y) = \bigcap \{ \beta_{\mathbf{B}}(f(a))\colon a \in A \text{ and } Y \subseteq \beta_{\mathbf{A}}(a) \}$,
\item $\displaystyle f^\sigma(V) = \bigvee \{ f^\sigma(Y) \colon V \supseteq Y \in C_{\mathcal{K}_{\mathbf{A}}}(X(\mathbf{A})) \}$, 
\item $\displaystyle f^\pi(Z^{c}) = \bigvee \{\beta_{\mathbf{B}}(f(a)) \colon a \in A \text{ and } \beta_{\mathbf{A}}(a) \subseteq Z^{c} \}$,
\item $\displaystyle f^\pi(V) = \bigcap \{ f^\pi(Z^{c}) \colon V^{c} \supseteq Z \in \mathcal{Z}(X(\mathbf{A})) \}$.
\end{enumerate}
\end{lemma}

\begin{lemma} \label{lema auxiliar}
Let $\mathbf{A}$ and $\mathbf{B}$ be two semilattices and $f \colon A \to B$ be an order preserving map. Then for every $Z \in \mathcal{Z}(X(\mathbf{A}))$ the subset $\{ \beta_{\mathbf{B}}(f(a)) \colon a \in I_{\mathbf{A}}(Z)\}$ is directed.
\end{lemma}

\begin{proof}
It follows from the fact that $I_{\mathbf{A}}(Z)$ is directed and both $f$ and $\beta_{\mathbf{B}}$ are monotone.
\end{proof}

\begin{lemma} \label{previo relacion}
Let $\mathbf{A}$ and $\mathbf{B}$ be two semilattices and $f \colon A \rightarrow B$ be an order preserving map. Then for every $Z \in \mathcal{Z}(X(\mathbf{A}))$ we have 
\begin{equation*}
f^\pi(Z^{c}) = \{ P \in X(\mathbf{B}) \colon f^{-1}[P] \cap I_{\mathbf{A}}(Z) \neq \emptyset \}.
\end{equation*}
\end{lemma}

\begin{proof}
It follows from Lemma \ref{Topological extensions of maps}.
\end{proof}

Let $\mathbf{A}$ and $\mathbf{B}$ be two semilattices and $f \colon A \rightarrow B$ be an order preserving map. We define $R_{f} \subseteq X(\mathbf{B}) \times \mathcal{Z}(X(\mathbf{A}))$ by 
\begin{equation*}
(P,Z) \in R_{f} \Longleftrightarrow f^{-1}[P] \cap I_{\mathbf{A}}(Z) = \emptyset.
\end{equation*}

\begin{proposition} \label{Descripcion fpi}
Let $\mathbf{A}$ and $\mathbf{B}$ be two semilattices and $f \colon A \rightarrow B$ be an order preserving map. Then for every $V \in E(X(\mathbf{A}))$ we have 
\begin{equation*}
f^{\pi}(V) = \{ P\in X(\mathbf{B}) \colon \forall Z \in R_{f}(P)  [Z \cap V \neq \emptyset] \}.
\end{equation*}
\end{proposition}

\begin{proof}
Let $V \in E(X(\A))$. In order to prove the first inclusion, let $P \in f^{\pi}(V)$ and $Z \in R_{f}(P)$. If $Z \cap V = \emptyset$, then $Z \subseteq V^{c}$. By Lemma \ref{Topological extensions of maps}, it follows that $P\in f^{\pi}(Z^{c})$. So, by Lemma \ref{previo relacion}, $f^{-1}[P] \cap I_{\A}(Z) \neq \emptyset$ which contradicts our assumption. Hence, $Z\cap V\neq\emptyset$. For the remaining inclusion, let $P \in  X(\mathbf{B})$ be such that for every $Z \in R_{f}(P)$ we have $Z \cap V \neq \emptyset$. Let $Z \in \mathcal{Z}(X(\A))$ be such that $Z \subseteq V^{c}$. We prove that $f^{-1}[P] \cap I_{\mathbf{A}}(Z) \neq \emptyset$. If we assume the contrary, then $Z \in R_{f}(P)$ and $Z \cap V \neq \emptyset$ which gives us a contradiction. Therefore, $P \in f^{\pi}(V)$.
\end{proof}

Finally, we analyze $f^\sigma$. Note that for every $P \in X(\mathbf{B})$ and every $Y \in C_{\mathcal{K}_{\mathbf{A}}}(X(\mathbf{A}))$ we have $P \in f^{\sigma}(Y)$ if and only if $\psi(Y) \subseteq f^{-1}[P]$ where $\psi(Y) = \{a \in A \colon Y \subseteq \beta_\A(a)\}$. So, we can define the relation $G_{f} \subseteq X(\mathbf{B}) \times C_{\mathcal{K}_{\mathbf{A}}}(X(\mathbf{A}))$ by 
\begin{equation*}
(P,Y) \in G_{f} \Longleftrightarrow \psi(Y) \subseteq f^{-1}[P].
\end{equation*}

\begin{proposition}\label{sigma estension}
Let $\mathbf{A}$ and $\mathbf{B}$ be two semilattices and $f \colon A \rightarrow B$ be an order preserving map and let $\Lambda$ be the closure operator defined in (\ref{operador lambda}). Then for every $V \in E(X(\mathbf{A}))$ we have 
\begin{equation*}
f^{\sigma}(V) = \Lambda \left( \bigcup \{ G_{f}^{-1}(Y) \colon V \supseteq Y \in C_{\mathcal{K}_{\mathbf{A}}}(X(\mathbf{A})) \} \right).
\end{equation*}
\end{proposition}

\begin{proof}
Let $V \in E(X(\mathbf{A}))$ and $P \in f^{\sigma}(P)$. Since $f^{\sigma}(Y) = G_{f}^{-1}(Y)$ for all $Y \in C_{\mathcal{K}_{\mathbf{A}}}(X(\mathbf{A}))$, we have 
\begin{equation*}
\bigcup \{ f^{\sigma}(Y) \colon V \supseteq Y \in C_{\mathcal{K}_{\mathbf{A}}}(X(\mathbf{A})) \} = \bigcup \{ G_{f}^{-1}(Y) \colon V \supseteq Y \in C_{\mathcal{K}_{\mathbf{A}}}(X(\mathbf{A})) \}
\end{equation*}
so the result follows.
\end{proof}

Unlike the case of monotone distributive semilattices, in where the topological descriptions of both extensions of a monotone map are quite easy to handle (see \cite{menchon2} for details), from Proposition \ref{sigma estension} it can be noticed that in practice, the topological description of the $\sigma$ extension of a monotone map $f$ between semilattices results more cumbersome to work with than its $\pi$ extension. Motivated by this fact, in the next section we will develop our duality by means of the topological description of $f^{\pi}$.

\section{Topological duality} \label{duality}

As we saw along Section \ref{section: Ideals}, the topological description of canonical extensions of semilattices can be used to obtain topological descriptions of the canonical extensions of monotone maps of semilattices, namely, the extensions $\pi$ and $\sigma$. In this section we take advantage of these results in order to define a suitable category of multirelational topological spaces which will serve as the categorical dual of the category of monotone semilattices. The construction of such multirelations strongly relies in a modeling process raising from the $f^{\pi}$ extension of a monotone function $f$ between semilattices.    

\begin{definition} 
Let $\A$ be a semilattice. A map $m \colon A \rightarrow A$ is said to be \emph{monotone} if it is an order preserving map, i.e, it satisfies the following condition: if $a \leq b$, then $ma \leq mb$ for all $a \in A$. Then a \emph{monotone semilattice} is a pair $\langle \A, m\rangle$ where $\A$ is a semilattice and $m$ is a monotone operation on $A$.
\end{definition}

\begin{remark} \label{Monotonic semilattices are a variety}
It is immediate that for every semilattice $\A$ a map $m \colon A \rightarrow A$ is monotone if and only if $m( a \wedge b) \leq ma \wedge mb$ for all $a,b \in A$. 
\end{remark}

It is clear from Remark \ref{Monotonic semilattices are a variety} that monotone semilattices form a variety. We write $\MS$ for the variety of monotone semilattices. Let $\langle \mathbf{A}, m \rangle$ and $\langle\mathbf{B}, n \rangle$ be two monotone semilattices. We say that a homomorphism $h \colon A\rightarrow B$ is a \emph{monotone homomorphism} if $h(ma)=nh(a)$ for all $a \in A$. We write $\mathsf{mMS}$ for the algebraic category of monotone semilattices and monotone homomorphisms.

A \emph{multirelation} on set $X$ is a subset of the cartesian product $X \times \mathcal{P}(X)$, that is, a set of ordered pairs $(x,Y)$ where $x\in X$ and $Y \subseteq X$ (\cite{Duntsch-Orlowska-RewitzkyMultirelations2010,Rewitzky2003}). Recall that in classical monotone modal logic a neighborhood frame is a pair $\langle X, R \rangle$ where $X$ is a set and $R \subseteq X \times \mathcal{P}(X)$, i.e., $R$ is a multirelation (\cite{Chellas,Hansen}).

We will represent a monotone operator $m$ on a semilattice $\A$ as a multirelation on the dual space of $\A$, where the canonical extension offers an advantageous point of view. To do so, we will apply the results we obtained in Section \ref{Extension of monotonic maps}.

From Lemmas \ref{Topological extensions of maps}, \ref{lema auxiliar} and Lemma \ref{lema supremo} we obtain a topological description of the monotone operators on a semilattice.

\begin{lemma} \label{Extensions topologically}
Let $\langle \mathbf{A}, m \rangle\in \MS$.	Then $\langle \Can{A}, m^{\pi} \rangle \in \MS$ and for every $a \in A$, $m^{\pi}(\beta(a)) = \beta(m(a))$. Moreover, for every $Z \in \mathcal{Z}(X(\mathbf{A}))$ and $V \in E(X(\A))$ the following hold:
\begin{enumerate}
\item $\displaystyle m^\pi(Z^{c}) = \bigcup \{\beta(ma) \colon Z \subseteq \beta(a)^{c} \}$,
\item $\displaystyle m^\pi(V) = \bigcap \{ m^\pi(Z^{c}) \colon V^{c} \supseteq Z \in \mathcal{Z}(X(\mathbf{A})) \}$.
\end{enumerate}
\end{lemma}

Let $\langle \mathbf{A}, m \rangle\in \MS$. We consider the multirelation $R_{m} \subseteq X(\mathbf{A}) \times \mathcal{Z}(X(\mathbf{A}))$ given by  \begin{equation*}
(P,Z) \in R_{m} \Longleftrightarrow m^{-1}[P] \cap I_{\mathbf{A}}(Z) = \emptyset.
\end{equation*}

As a particular case of Proposition \ref{Descripcion fpi} we obtain the following result.

\begin{proposition} \label{Extensions frames}
Let $\langle \mathbf{A}, m \rangle\in \MS$. Then for every $V \in E(X(\mathbf{A}))$ we have 
\begin{equation*}
m^{\pi}(V) = \{ P \in X(\mathbf{A}) \colon \forall Z \in R_{m}(P)  [Z \cap V \neq \emptyset] \}.
\end{equation*}
\end{proposition}

Let $\langle X, \mathcal{K} \rangle$ be an $S$-space. For each $U \in S(X)$ we define the subset $L_{U}$ of $\mathcal{Z}(X)$ as follows:
\begin{equation*}
L_{U} = \{ Z \in \mathcal{Z}(X) \colon Z \cap U \neq \emptyset \}. 
\end{equation*}

\begin{definition} \label{t esp mon} 
An \emph{$mS$-space} is a structure $\langle X,\mathcal{K},R\rangle$ where $\langle X, \mathcal{K}\rangle$ is an $S$-space and $R \subseteq X \times  \mathcal{Z}(X)$ a multirelation such that: 
\begin{enumerate}
\item $m_{R}(U) = \{ x\in X \colon \forall Z \in R(x) [ Z \cap U \neq \emptyset] \}\in S(X)$ for all $U \in S(X)$, 
\item $R(x) = \displaystyle \bigcap \{L_{U} \colon U \in S(X) \text{ and } x \in m_{R}(U) \}$ for all $x \in X$.
\end{enumerate}
\end{definition}

Notice that condition 1. of Definition \ref{t esp mon} can be restated as $$m_{R}(U) = \{ x \in X \colon R(x) \subseteq L_{U} \}\in S(X)$$ for all $U\in S(X)$. In addition, condition 2. says that for every $mS$-space $\langle X, \mathcal{K}, R \rangle$, the map $m_{R} \colon S(X) \rightarrow S(X)$ is well defined and it is monotone.

Let $\mathsf{Rel}$ be the category of sets and binary relations and $\mathsf{Set}$ be the category of sets and functions. It is well known that there exists a faithful functor $\mathcal{F} \colon \mathsf{Rel}^{op} \rightarrow \mathsf{Set}$ defined for every set $X$ by $\mathcal{F}(X) = \mathcal{P}(X)$, and for every binary relation $T \subseteq X \times Y$ and every $U \in \mathcal{P}(Y)$, $\mathcal{F}(T)$ is defined by
\begin{equation*}
\mathcal{F}(T)(U) = \{ x \in X \colon T(x) \subseteq U\}.
\end{equation*}

A \emph{meet-relation} between two $S$-spaces $\langle X_{1}, \mathcal{K}_{1} \rangle$ and $\langle X_{2}, \mathcal{K}_{2} \rangle$ was defined in \cite{Celani2020} as a relation $T \subseteq X_{1} \times X_{2}$ satisfying the following conditions:
\begin{enumerate}
\item $\mathcal{F}(T)(U) \in S(X_{1})$ for all $U \in S(X_{2})$, 
\item $T(x) = \displaystyle \bigcap \{ U \in S(X_{2}) \colon T(x) \subseteq U \}$ for all $x \in X_{1}$.
\end{enumerate}

Let $\langle X_{i},\mathcal{K}_{i}\rangle$, with $i=1,2,3$ be $S$-spaces and $R \subseteq X_{1} \times X_{2}$ and $T \subseteq X_{2} \times X_{3}$ be meet-relations. Consider now the following relation 
\begin{equation*}
T \ast R = \{ (x,z) \in X_{1} \times X_{3} \colon \forall U\in S(X_{3}) [(T\circ R)(x)\subseteq U \Rightarrow z\in U ]\}.
\end{equation*}
In \cite{CelaniCalomino,Celani2020} it was proved that $T \ast R$ is a meet-relation, $\ast$ is associative and if $\sqsupseteq_{2}$ is the dual of the specialization order of $X_{2}$, then $T \ast \sqsupseteq_{2} = T$ and $\sqsupseteq_{2} \ast R =R$. Observe that the latter can be rephrased by means that $S$-spaces and meet-relations form a category in which the composition is given by $\ast$ and the identity arrow is given by the dual of the specialization order. We write $\mathsf{mRel}$ for such a category. Let $\mathsf{MS}$ be the category of semilattices and homomorphisms. Let us assign to each $S$-space $\langle X, \mathcal{K} \rangle$ the semilattice $S(X)$ and to each meet-relation $T$ the map $\mathcal{F}(T)$. In Proposition 3.16 of \cite{Celani2020} it was proved that $\mathcal{F}(T \ast R) = \mathcal{F}(T\circ R)$ and it is no hard to see that $\mathcal{F}(\sqsupseteq)$ is the identity map of $S(X)$. Moreover, from the definition of meet-relation, $\mathcal{F}(T)$ and Proposition 3.19 of \cite{Celani2020}, it follows that these assignments  define a functor $\Box \colon \mathsf{mRel}^{op} \rightarrow \mathsf{MS}$ which is faithful by Lemma 3.21 of \cite{Celani2020}. We stress that this is an important operational advantage taking into account that $(T\ast R)(x) = \mathrm{cl}_{\mathcal{K}_{3}}((T\circ S)(x))$ for all $x \in X_{1}$. In the literature (\cite{CelaniCalomino,Celani2020,menchon2}) $\Box(T)$ is denoted by $\Box_{T}$. So for the sake of readability we will keep such a notation along this paper.

\begin{definition} \label{monotonic meet relations}
Let $\langle X_{1}, \mathcal{K}_{1}, R_{1} \rangle$ and $\langle X_{2}, \mathcal{K}_{2}, R_{2} \rangle$ be two $mS$-spaces. A meet-relation $T \subseteq X_{1} \times X_{2}$ is a \emph{monotone meet-relation} if the following diagram commutes:
\begin{displaymath}
\xymatrix{
S(X_{2}) \ar[r]^-{\Box_{T}} \ar[d]_-{m_{R_{2}}} & S(X_{1})\ar[d]^-{m_{R_{1}}}
\\
S(X_{2}) \ar[r]_-{\Box_{T}} & S(X_{1})
}
\end{displaymath}
\end{definition}

The following result is a topological characterization of monotone meet-relations. Its proof is essentially the same as that of Lemma 44 in \cite{menchon2} so we will omit it.

\begin{proposition} \label{meet-relations topolical}
Let $\langle X_{1}, \mathcal{K}_{1}, R_{1} \rangle$ and $\langle X_{2}, \mathcal{K}_{2}, R_{2}\rangle$ be two $mS$-spaces. A meet-relation $T \subseteq X_{1} \times X_{2}$ is monotone if and only if for every $x \in X_{1}$ and every $U \in S(X_{2})$ we have 
\begin{equation*}
U^{c} \in R_{2}[T(x)] \Longleftrightarrow T^{-1}[U^{c}] \in R_{1}(x) 
\end{equation*}
where $R_{2}[T(x)] = \{ Z \in \mathcal{Z}(X_{2}) \colon \exists y \in T(x) [(y,Z) \in R_{2}] \}$.
\end{proposition}

\begin{proposition} \label{composition mmeet-relations}
Let $\langle X_{1}, \mathcal{K}_{1}, R_{1} \rangle$, $\langle X_{2}, \mathcal{K}_{2}, R_{2} \rangle$ and $\langle X_{2}, \mathcal{K}_{3}, R_{3}\rangle$ be $mS$-spaces. Let $L \subseteq X_{1}\times X_{2}$ and $T \subseteq X_{2} \times X_{3}$ be monotone meet-relations. Then $T \ast L$ is a monotone meet-relation. Moreover, $\ast$ is associative.
\end{proposition}

\begin{proof}
Consider the diagram
\begin{displaymath}
\xymatrix{
S(X_{3}) \ar[r]^-{\Box_{T}} \ar[d]_-{m_{R_{3}}}  & S(X_{2})\ar[d]^-{m_{R_{2}}} \ar[r]^-{\Box_{L}} & S(X_{1}) \ar[d]^-{m_{R_{1}}}
\\
S(X_{3}) \ar[r]_-{\Box_{T}}  & S(X_{2}) \ar[r]_-{\Box_{L}} & S(X_{1})
}
\end{displaymath}
Since $\Box$ is contravariant and each of the squares is commutative by assumption, then from the following calculation
\begin{equation*}
m_{R_{1}} \Box_{T\ast L} = m_{R_{1}}(\Box_{L} \Box_{T}) = (\Box_{L} m_{R_{2}}) \Box_{T} = \Box_{L}(\Box_{T} m_{R_{3}}) = \Box_{T\ast L} m_{R_{3}}
\end{equation*}
we obtain that $T \ast L$ is a meet-relation. The proof of the associativity is analogue.
\end{proof}

\begin{proposition} \label{identity of mMS-spaces}
Let $\langle X, \mathcal{K}, R\rangle$ be an $mS$-space. The dual of the specialization order $\sqsupseteq \subseteq X \times X$ is a monotone meet-relation.
\end{proposition}

\begin{proof}
Recall that from the functoriality of $\Box$ it follows that $\Box_{\sqsupseteq}$ is in fact the identity map of $S(X)$. So, it is clear that $\sqsupseteq$ satisfies Definition \ref{monotonic meet relations}. 
\end{proof}

By Propositions \ref{composition mmeet-relations} and \ref{identity of mMS-spaces} we get that $mS$-spaces and monotone meet-relations form a category in which the identity map is given by the dual of the specialization order. We write $\mathsf{mSp}$ for such a category. 

\begin{lemma} \label{homeomorphisms induce isomorphisms}
Let $\langle X_{1}, \mathcal{K}_{1}, R_{1} \rangle$ and $\langle X_{2}, \mathcal{K}_{2}, R_{2} \rangle$ be two $mS$-spaces and $f \colon X_{1}\rightarrow X_{2}$ be a bijective map such that
\begin{enumerate}
\item $\mathcal{K}_{2} = \{ f[V] \colon V \in \mathcal{K}_{1} \}$,
\item $(x,Z) \in R_{1}$ if and only if $(f(x),f[Z]) \in R_{2}$ for all $x \in X_{1}$ and all $Z \in \mathcal{Z}(X_{1})$.
\end{enumerate}
Then $f$ induces an isomorphism of $mS$-spaces.
\end{lemma}

\begin{proof}
Since $f$ is bijective, we can consider the relations $R_{f} \subseteq X_{1} \times X_{2}$ and $T_{f} \subseteq X_{2} \times X_{1}$ which are given by
\begin{equation*}
(x,f(y)) \in R_{f} \Longleftrightarrow f(x) \sqsupseteq_{2} f(y)
\end{equation*}
and 
\begin{equation*}
(f(y),x) \in T_{f} \Longleftrightarrow f(y) \sqsupseteq_{2} f(x).
\end{equation*}
We prove that $R_{f}$ and $T_{f}$ are monotone-meet relations. We only exhibit the proof for the case of $R_{f}$ because the proof for $T_{f}$ is similar. Since $f$ is bijective and from 1. we obtain $\mathcal{Z}(X_{2}) = \{f[Z] \colon Z \in \mathcal{Z}(X_{1})\}$, $S(X_{2})=\{f[W]\colon W\in S(X_{1})\}$ and $\Box_{R_{f}}(f[U]) = U$, for every $U \in S(X_{1})$. Whence, it follows that for every $x,y \in X_{1}$, $f(x) \sqsupseteq_{2} f(y)$ if and only if $x \sqsupseteq_{1} y$. Notice that from these facts it is straightforward to see that $R_{f}$ is a meet-relation. In order to proof that $R_{f}$ is monotone we need to check that the following diagram commutes: 
\begin{displaymath}
\xymatrix{
S(X_{2}) \ar[r]^-{\Box_{R_{f}}} \ar[d]_-{m_{R_{2}}} & S(X_{1})\ar[d]^-{m_{R_{1}}}
\\
S(X_{2}) \ar[r]_-{\Box_{R_{f}}} & S(X_{1})
}
\end{displaymath}
Now we prove that $m_{R_{1}}(U) = \Box_{R_{f}}(m_{R_{2}}(f[U]))$ for all $U \in S(X_{1})$. Let $x \in \Box_{R_{f}}(m_{R_{2}}(f[U]))$, then $R_{f}(x)\subseteq m_{R_{2}}(f[U])$. Thus, for every $y \in X_{1}$ such that $f(x) \sqsupseteq_{2} f(y)$ we have $f(y) \in m_{R_{2}}(f[U])$. Observe that for every $Z \in \mathcal{Z}(X_{1})$ such that $(f(y),f[Z]) \in R_{2}$, we have $f[Z] \cap f[U] \neq \emptyset$. Hence, from 2. and since $f$ is bijective by assumption, we get $y \in m_{R_{1}}(U)$. I.e., for every $y \in X_{1}$, if $x \sqsupseteq_{1} y$, then $y \in m_{R_{1}}(U)$. But this is equivalent to say that $x \in m_{R_{1}}(U)$. Thus, $\Box_{R_{f}}(m_{R_{2}}(f[U])) \subseteq m_{R_{1}}(U)$. The proof of the reverse inclusion is analogue.

Finally, we prove that $T_{f}$ is the inverse of $R_{f}$ in $\mathsf{mSp}$. It is clear from the definitions of $R_{f}$ and $T_{f}$ that $R_{f} \circ T_{f} =\sqsupseteq_{2}$ and $T_{f} \circ R_{f} = \sqsupseteq_{1}$. Then, since $\sqsupseteq_{2}$ is a meet-relation, for every $y \in X_{1}$ we have
\begin{equation*}
(R_{f}\ast T_{f})(f(y))=\mathrm{cl}_{\mathcal{K}_{2}}((R_{f}\circ T_{f})(f(y)))=\mathrm{cl}_{\mathcal{K}_{2}}(\sqsupseteq_{2}(f(y)))=\sqsupseteq_{2}(f(y)).
\end{equation*}
Therefore, $R_{f}\ast T_{f}=\sqsupseteq_{2}$. The proof for $T_{f}\ast R_{f}=\sqsupseteq_{1}$ is similar.
\end{proof}

\begin{corollary} \label{homeomorphisms induced by mS-spaces}
Let $\langle X_{1}, \mathcal{K}_{1}, R_{1}\rangle$ be an $mS$-space and $\langle X_{2}, \mathcal{K}_{2} \rangle$ be an $S$-space. If $f \colon X_{1}\rightarrow X_{2}$ is a bijective map such that $\mathcal{K}_{2} = \{ f[V] \colon V \in \mathcal{K}_{1} \}$, then there exists a multirelation $R_{2} \subseteq X_{2} \times \mathcal{Z}(X_{2})$ such that $\langle X_{1},\mathcal{K}_{1}, R_{1}\rangle$ and $\langle X_{2}, \mathcal{K}_{2}, R_{2}\rangle$ are isomorphic $mS$-spaces.
\end{corollary}

\begin{proof}
We start by noticing that from the hypothesis on $f$ it is immediate that $f[U] \in S(X_{2})$ for all $U \in S(X_{1})$ and $f[Z] \in \mathcal{Z}(X_{2})$ for all $Z \in \mathcal{Z}(X_{1})$. Therefore, it easily follows that $L_{f[U]} = f[L_{U}]$ for all $U \in S(X_{1})$. In addition, it is also the case that for every $V \in S(X_{2})$ there exists a unique $U \in S(X_{1})$ such that $V = f[U]$. We define $R_{2} \subseteq X_{2} \times \mathcal{Z}(X_{2})$ as  
\begin{equation*}
(f(x),W) \in R_{2} \Longleftrightarrow (x,f^{-1}[W]) \in R_{1}.
\end{equation*}
From the bijectivity of $f$ and the facts of above it follows that $R_{2}$ is well defined. In order proof that $\langle X_{2}, \mathcal{K}_{2}, R_{2} \rangle$ is an $mS$-space, note that since $f^{-1}[f[Z]] = Z$ for all $Z \in R_{1}(x)$, then we get $f[R_{1}(x)] = R_{2}(f(x))$ for all $x \in X$. It is no hard to see that $m_{R_{1}}(U) = f^{-1}[m_{R_{2}}(f[U])]$ for all $U \in S(X_{1})$. If $V \in S(X_{2})$, then
\begin{equation} \label{equation coro}
m_{R_{2}}(V) = m_{R_{2}}(f[U]) = f[m_{R_{1}}(U)].
\end{equation} 
Since $\langle X_{1}, \mathcal{K}_{1}, R_{1}\rangle$ is an $mS$-space, then $m_{R_{1}}(U) \in S(X_{1})$ and from (\ref{equation coro}) we conclude $m_{R_{2}}(V) \in S(X_{2})$. This proofs 1. of Definition \ref{t esp mon}. On the other hand, as $\langle X_{1}, \mathcal{K}_{1}, R_{1} \rangle$ is an $mS$-space and $f$ is bijective, we get 
\begin{equation*}
\begin{array}{ccl}
R_{2}(f(x)) & = & f[R_{1}(x)] \\
                & = & \displaystyle f \left[ \bigcap\{L_{U} \colon U \in S(X_{1}) \text{ and } x \in m_{R_{1}}(U) \} \right] \\
                & = & \displaystyle \bigcap \{ L_{f[U]} \colon U \in S(X_{1}) \text{ and } f(x) \in m_{R_{2}}(f[U]) \}. \\
\end{array}
\end{equation*}
So, 2. of Definition \ref{t esp mon} holds. Finally, it is clear from the definition of $R_{2}$ that condition 2. of Lemma \ref{homeomorphisms induce isomorphisms} holds. Therefore, $\langle X_{1}, \mathcal{K}_{1}, R_{1}\rangle$ and $\langle X_{2}, \mathcal{K}_{2}, R_{2}\rangle$ are isomorphic $mS$-spaces, as granted.
\end{proof}

We recall that if $\langle X, \mathcal{K} \rangle$ is an $S$-space, then we have that $\langle X(\mathbf{S}(X)), \mathcal{K}_{\mathbf{S}(X)} \rangle$ is the dual $S$-space associated to the semilattice $\mathbf{S}(X)$. In \cite{Celani2020} it was proved that the map 
\begin{equation*}
H_{X} \colon X \rightarrow X(\mathbf{S}(X))
\end{equation*}
given by $H_{X}(x) = \{ U \in S(X) \colon x\in U\}$ is a homeomorphism between $S$-spaces such that $\mathcal{K}_{\mathbf{S}(X)} = \{ H_{X}[U] \colon U \in \mathcal{K}\}$. Therefore, as a straightforward consequence of Corollary \ref{homeomorphisms induced by mS-spaces} we obtain:

\begin{lemma} \label{HX}
Let $\langle X, \mathcal{K}, R \rangle$ be a $mS$-space. Then the map $H_{X} \colon X \rightarrow X(\mathbf{S}(X))$ defined by $H_{X}(x) = \{ U \in S(X) \colon x \in U \}$ induces an isomorphism of $mS$-spaces. 
\end{lemma}

The proofs of the following two Lemmas are similar to Propositions 30 and 46 of \cite{menchon2} with the only difference that in our proofs, Lemma \ref{Extensions topologically} and Proposition \ref{Extensions frames} are involved. We leave the details to the reader.

\begin{lemma} \label{prop dual monotonic space} 
Let $\langle \mathbf{A}, m \rangle \in \mathcal{MS}$. Then $\langle X({\mathbf{A}}), \mathcal{K}_{\mathbf{A}}, R_{m} \rangle$ is an $mS$-space. Moreover, $m_{R_{m}} = m^{\pi}$.
\end{lemma}

Let $\mathbf{A}$ and $\mathbf{B}$ be two semilattices. Let $h \colon A \rightarrow B$ be a homomorphism. In \cite{Celani2020} it was proved that the binary relation $R_{h} \subseteq X({\mathbf{B}}) \times X({\mathbf{A}})$ defined by 
\begin{equation*}
(P,Q) \in R_{h} \Longleftrightarrow h^{-1}[P] \subseteq Q
\end{equation*}
is a meet-relation such that $\beta_{\mathbf{B}} h=\Box_{R_{h}} \beta_{\mathbf{A}}$.

\begin{lemma} \label{homo cond} 
Let $\langle \mathbf{A}, m \rangle, \langle \mathbf{B}, n \rangle \in \mathcal{MS}$ and $h \colon A \rightarrow B$ be a homomorphism. Then $h$ is a monotone homomorphism if and only if $R_{h}$ is a monotone meet-relation. In particular, if $h = id_{\mathbf{A}}$, then $R_{h}=\sqsupseteq_{X(\mathbf{A})}$.
\end{lemma}

\begin{lemma} \label{functorialidad}
Let $\langle \mathbf{A}, m \rangle, \langle \mathbf{B}, n \rangle, \langle \mathbf{C}, l \rangle \in \mathcal{MS}$. If $h \colon A \rightarrow B$ and $g \colon B \rightarrow C$ be monotone homomorphisms, then $R_{gh} = R_{h} \ast R_{g}$.
\end{lemma}

\begin{proof}
From Lemma \ref{homo cond} $R_{gh}$, $R_{h}$ and $R_{g}$ are monotone meet-relations. Then, since $\Box$ is faithful, for proving our claim it is enough to check $\Box_{R_{gh}} = \Box_{R_{h} \circ R_{g}}$. But this is essentially the Proposition 3.22 of \cite{Celani2020}.  
\end{proof} 

Notice that Lemmas \ref{prop dual monotonic space} and \ref{homo cond} allow to define a functor $G \colon \mathsf{mMS}^{op} \rightarrow \mathsf{mSp}$ as follows:
\begin{displaymath}
\begin{array}{ccc}
\langle\mathbf{A},m\rangle & \mapsto & \langle X({\mathbf{A}}), \mathcal{K}_{\mathbf{A}}, R_{m} \rangle \\
h                                      & \mapsto & R_{h}
\end{array}
\end{displaymath} 
On the other hand, observe that the assignments 
\begin{displaymath}
\begin{array}{ccc}
\langle X, \mathcal{K}, R \rangle & \mapsto & \langle S(X),m_{R}\rangle \\
R                                             & \mapsto & \Box_{R}
\end{array}
\end{displaymath} 
also define a functor $H \colon \mathsf{mSp} \rightarrow \mathsf{mMS}^{op}$. 

Now we state our main theorem.

\begin{theorem}\label{duality theorem}
The categories $\mathsf{mSp}$ and $\mathsf{mMS}$ are dually equivalent.
\end{theorem}

\begin{proof}
Let $\langle \mathbf{A}, m \rangle \in \mathcal{MS}$. From Theorem 3.24 of \cite{Celani2020} and Lemma \ref{Extensions topologically} it follows that the map $\beta \colon A \rightarrow S(X({\A}))$ defines a natural isomorphism between $H \circ G$ and $Id_{\mathsf{mMS}^{op}}$. On the other hand, let $\langle X, \mathcal{K}, R \rangle$ be a $mS$-space. Then, from Lemma \ref{HX} and again from Theorem 3.24 of \cite{Celani2020}, we get that the map $H_{X} \colon X \rightarrow X(\mathbf{S}(X))$ leads to a natural isomorphism between $G \circ H$ and $Id_{\mathsf{mSp}}$. This concludes the proof.  
\end{proof}

\section{Characterization of homomorphic images and congruences by means of lower-Vietoris-type topologies}\label{homomorphic images}

In \cite{CelaniCalomino} Celani and Calomino proved that homomorphic images of a distributive semilattice are characterized by one-to-one meet-relations and by lower Vietoris families on the dual space. This characterization was motivated by the results given in \cite{Guram-Jansana II} and by the fact that the class of distributive semilattices is not a variety. Of course, this is not the case of semilattices and monotone semilattices. In this section we characterize congruences of semilattices and monotone semilattices by means of lower-Vietoris-type topologies and monotone lower-Vietoris-type topologies of their respective associated dual spaces. To achieve that goal, we start by studying the existing relation between homomorphic images and lower-Vietoris-type topologies in the variety of semilattices and later we extend this study to the variety of monotone semilattices.

\begin{definition}
Let $\langle X_1, \mathcal{K}_{1} \rangle$ and $\langle X_2, \mathcal{K}_{2} \rangle$ be two $S$-spaces. A meet-relation $R \subseteq X_1 \times X_2$ is \emph{one-to-one} if for each $x \in X_1$ and $U \in S(X_1)$ with $x \notin U$, there exists $V \in S(X_2)$ such that $U \subseteq \Box_R(V )$ and $x \notin \Box_R(V )$. 
\end{definition}

The following result is the $S$-space version of Theorem 36 (2) of \cite{CelaniCalomino}. Since its proof is similar, we omit it.

\begin{theorem}\label{onto one-to-one}
Let $\langle X_1, \mathcal{K}_{1} \rangle$ and $\langle X_2, \mathcal{K}_{2} \rangle$ be two $S$-spaces and $R \subseteq X_1 \times X_2$ be a meet-relation. Then the homomorphism $\Box_R \colon S(X_2) \rightarrow S(X_1)$ is onto if and only if $R$ is one-to-one.
\end{theorem}

\begin{corollary}
Let $\A$ and $\mathbf{B}$ be two semilattices and $h \colon A \rightarrow B$ be a homomorphism. Then $h$ is onto if and only if the meet-relation $R_h \subseteq X(\mathbf{B}) \times X(\A)$ is one-to-one. 
\end{corollary}

Let $\A$ be a semilattice and $\langle X, \K \rangle$ be an $S$-space. Consider a one-to-one meet-relation $R \subseteq X \times X(\A)$ and let $\mathcal{F}_R = \{R(x) \colon x \in X\}$. It is clear that $\mathcal{F}_R \subseteq C_{\K}(X(\A))$. For every $a \in A$ we consider the set
\begin{equation*}
H_a = \{ R(x) \colon R(x) \cap \beta (a)^c \neq \emptyset \}.
\end{equation*}

\begin{remark}\label{basecong}
Notice that for every $a \in A$ and $x \in X$ the following holds 
\begin{equation*}
R(x) \in H_a  \Longleftrightarrow x \in \Box_R(\beta(a))^c.
\end{equation*}
\end{remark}

In \cite{Ivanova-dimova}, Ivanova-Dimova introduces a new lower-Vietoris-type hypertopology that is a generalized version of the lower Vietoris topology.

\begin{definition} (\cite{Ivanova-dimova}) 
Let $X$ be a set. Let $\mathcal{F}$ be a family of non-empty subsets of $X$ and $\mathcal{O}$ be a topology on $\mathcal{F}$. For each $U\subseteq X$ let us consider 
\begin{equation*}
U^-_\mathcal{F} = \{ Y \in \mathcal{F} \colon Y \cap U \neq \emptyset \}.
\end{equation*}
Let $\mathcal{M}_{\mathcal{F}}=\{U^-_\mathcal{F} \colon U \subseteq X \}$. We say that $\mathcal{O}$ is a \emph{lower-Vietoris-type topology} on $\mathcal{F}$ if $\mathcal{O} \cap \mathcal{M}_{\mathcal{F}}$ is a subbase for $\mathcal{O}$.
\end{definition}

\begin{proposition} \label{subbasis lower type Vietoris}
Let $\A$ be a semilattice and $\langle X, \K \rangle$ be an $S$-space. Let $R \subseteq X \times X(\A)$ be a one-to-one meet-relation. Then:
\begin{enumerate}
\item $R(x) \neq \emptyset$ for all $x \in X$.
\item The family $\mathcal{M} = \{H_a \colon a \in A\}$ is a subbasis for a topology on $\mathcal{F}_R$.
\item The topology $\tau_\mathcal{M}$ generated by $\mathcal{M}$ is a lower-Vietoris-type topology on $\mathcal{F}_R$.
\end{enumerate}
\end{proposition}

\begin{proof}
We only need to prove 1. since the proof of 2. is analogue to the proof of Lemma 38 of \cite{CelaniCalomino} and 3. is a consequence of 2. and Corollary 2.11 of \cite{Ivanova-dimova}. Let $x \in X$. Since $\K$ is a subbasis, there exists $U \in S(X)$ such that $x \notin U$. As $R$ is a one-to-one meet-relation, there exists $a \in A$ such that $U \subseteq \Box_R(\beta(a))$ and $x \notin \Box_R(\beta(a))$. Therefore, from Remark \ref{basecong}, $R(x) \in H_{a}$ and this implies that $R(x) \neq \emptyset$.
\end{proof}

\begin{remark}
We would like to stress that although the proof of Proposition \ref{subbasis lower type Vietoris} (2) is similar to the proof of Lemma 38 in \cite{CelaniCalomino} when considering S-spaces, the family $\mathcal{M}$ in general is not a basis. To check this, let us consider the following non-distributive lattice $\mathbf{L}$ together with the identity map $id_{\bf{L}} \colon L \rightarrow L$:
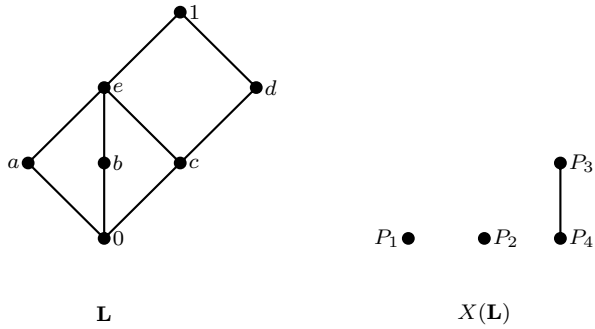
\label{example1}
\begin{figure}[H]
\begin{center}
\begin{tikzpicture}
\filldraw [thick]
(0,1) circle (2pt) node[left] {$a$} --
(1,0) circle (2pt) node[right](x) {$0$}     -- 
(1,1) circle (2pt) node[right](y) {$b$}--(1,2) 
(1,0)--(2,1) circle (2pt) node[right] {$c$} --
(1,2) circle (2pt) node[right] {$e$}--(0,1)
(2,1)--(3,2)circle (2pt) node[right] {$d$}--
(2,3)circle (2pt) node[right] {$1$}-- (1,2);
\draw (1,-1) node {$\mathbf{L}$};
\filldraw [thick]
(5,0) circle (2pt) node[left] {$P_1$} 
(6,0) circle (2pt) node[right] {$P_2$}     
(7,1) circle (2pt) node[right] {$P_3$} --
(7,0) circle (2pt) node[right] {$P_4$};
\draw (6,-1) node {$X(\mathbf{L})$};
\end{tikzpicture}
\label{ej1}
\end{center}
\end{figure}
It is no hard to see that $X(\mathbf{L}) = \{ P_{1}, P_{2}, P_{3}, P_{4} \}$, where $P_1 = \{ a,e,1 \}$, $P_2 = \{b,e,1\}$, $P_3 = \{c,e,d,1\}$ and $P_4=\{d,1\}$. It follows that $(P,Q) \in R_{id_{\mathbf{L}}}$ if and only if $P \subseteq Q$. Therefore, $\mathcal{F}_{R_{id_{\mathbf{L}}}} = \{ \{P_{1}\}, \{P_{2}\}, \{P_{3}\}, \{P_{3},P_{4}\}\}$. Notice that $H_{0} = \mathcal{F}_{R_{id_{\mathbf{L}}}}$, $H_{a} = \{ \{P_{2}\}, \{P_{3}\}, \{P_{3}, P_{4}\}\}$, $H_{b} = \{\{P_{1}\}, \{P_{3}\}, \{P_{3}, P_{4}\}\}$, $H_{c} = \{\{P_{1}\}, \{P_{2}\}, \{P_{3}, P_{4}\}\}$, $H_{d} = \{ \{P_{1}\}, \{P_{2}\} \}$, $H_{e} = \{ \{P_{3}, P_{4}\} \}$ and $H_{1}=\emptyset$. Nevertheless, $H_{a} \cap H_{d} = \{ \{P_{2}\} \} \notin \mathcal{M}$. Whence $\mathcal{M}$ is not a basis. 
\end{remark}

We omit the proof of the following Lemma because it is similar to the proof of Lemma 39 (3) in \cite{CelaniCalomino}.

\begin{lemma} \label{cerradoVietoris}
Let $\A$ be a semilattice and $\langle X, \K \rangle$ be an $S$-space. Let $R \subseteq X \times X(\A)$ be a one-to-one meet-relation and let $\varphi$ be the map defined in (\ref{definition of fi}). Then $Y \subseteq \mathcal{F}_R$ is a subbasic closed subset of $\langle \mathcal{F}_R, \mathcal{M} \rangle$ if and only if there exists $F \in \Fi(\A)$ such that $Y = \{ R(x) \colon R(x) \subseteq \varphi (F)\}$.
\end{lemma}

\begin{theorem} \label{theo-F space}
Let $\A$ be a semilattice and $\langle X, \K \rangle$ be an $S$-space. Let $R \subseteq X \times X(\A)$ be a one-to-one meet-relation. Then $\langle \mathcal{F}_R, \mathcal{M} \rangle$ is an $S$-space. 
\end{theorem}

\begin{proof}
We prove the conditions of Definition \ref{S-spaces}. By Proposition \ref{subbasis lower type Vietoris} (2) the family $\mathcal{M}$ is a subbasis for a topology on $\mathcal{F}_R$. We see that $\langle \mathcal{F}_R, \mathcal{M} \rangle$ is $T_0$. Let $x,y \in X$ be such that $R(x) \neq R(y)$. Since $\langle X, \K \rangle$ is an $S$-space, there exists $U \in S(X)$ such that $x \in U$ and $y \notin U$. As $R$ is a one-to-one meet-relation, there exists $a \in A$ such that $U\subseteq \Box_R(\beta(a))$ and $y \notin \Box_R(\beta(a))$. From Remark \ref{basecong} it follows that $R(y) \in H_a$ and $R(x) \notin H_a$. So, ($S_1$) is satisfied.

We see ($S_2$). It is clear that $H_1 = \emptyset$ and $\mathcal{M}$ is closed under finite unions. Let $a \in A$ and consider $H_a$. Observe that as a consequence of Remark \ref{basecong}, we have $H_a \subseteq \bigcup \{H_b \colon b \in B \subseteq A\}$ if and only if $\Box_R(\beta(a))^c \subseteq \bigcup \{\Box_R(\beta(b))^c \colon b \in B \subseteq A \}$. Since $\Box_R(\beta(a))^c$ is a compact subset, by Alexander's subbasis Theorem we get that $H_a$ is compact.

Notice that ($S_3$) follows from Remark \ref{basecong} and from the assumption that $\langle X, \K \rangle$ is an $S$-space.

Finally, we prove ($S_4$). Let $Y \in C_\mathcal{M}(\mathcal{F}_R)$ and $\mathcal{J} \subseteq S(\mathcal{F}_R)$ be a $Y$-family such that $Y \cap A^c \neq \emptyset$ for all $A \in \mathcal{J}$. We show that $Y \cap \bigcap \{A^c \colon A \in \mathcal{J} \} \neq \emptyset$. By Lemma \ref{cerradoVietoris} there exists $F \in \Fi(\A)$ such that $Y = \{R(x) \colon R(x) \subseteq \varphi(F)\}$. We consider  $G = \bigcap \{\Box_R(\beta(f)) \colon f \in F \}$ which is a subbasic closed subset of $S(X(\A))$ and $\mathcal{I} = \{\Box_R(\beta(a)) \colon H_{a}^{c} \in \mathcal{J} \}$. 

Let us see that $\mathcal{I}$ is a $G$-family of $S(X)$. Let $H_{a}^{c}, H_{b}^{c} \in \mathcal{J}$ be such that $\Box_R(\beta(a)), \Box_R(\beta(b)) \in \mathcal{I}$. Since $\mathcal{J}$ is a $Y$-family, there exist $d,e \in A$ such that $H_{e}^{c} \in \mathcal{J}$, $Y \subseteq H_{d}^{c}$, $H_{a}^{c} \cap H_{d}^{c} \subseteq H_{e}^{c}$ and $H_{b}^{c} \cap H_{d}^{c} \subseteq H_{e}^{c}$. From $Y \subseteq H_{d}^{c}$ we get that $\bigcap \{\Box_R(\beta(f)) \colon f \in F \} \subseteq \Box_R(\beta(d))$, i.e., $G \subseteq \Box_R(\beta(d))$. Moreover, $\Box_R(\beta(e)) \in \mathcal{I}$. By Remark \ref{basecong}, $\Box_R(\beta(a)) \cap \Box_R(\beta(d)) \subseteq \Box_R(\beta(e))$ and $\Box_R(\beta(b)) \cap \Box_R(\beta(d)) \subseteq \Box_R(\beta(e))$ and we get that $\mathcal{I}$ is a $G$-family. Now we prove $G \cap \bigcap \{A^c \colon A \in \mathcal{I} \} \neq \emptyset$. Suppose that $G \cap \bigcap \{ A^c \colon A \in \mathcal{I} \}= \emptyset$. So, there exists $H_{a}^{c} \in \mathcal{J}$ such that $G \cap \Box_R(\beta(a))^c = \emptyset$. Then $G \subseteq \Box_R(\beta(a))$ and again by Remark \ref{basecong} we have $Y \cap H_a = \emptyset$ which is a contradiction. So, $G \cap \bigcap \{A^c \colon A \in \mathcal{I} \} \neq \emptyset$. 

Suppose now that $Y \cap \bigcap \{A^c \colon A \in \mathcal{J} \} = \emptyset$. Then $Y \subseteq \bigcup \{H_{a}^{c} \colon H_{a}^{c} \in \mathcal{J} \}$. By Remark \ref{basecong} we get $\bigcap \{\Box_R(\beta(f)) \colon f \in F \} \subseteq \bigcup \{ \Box_R(\beta(a)) \colon H_{a}^{c} \in\mathcal{J} \}$, i.e., $G \subseteq \bigcup \{ A \colon A \in \mathcal{I} \}$ which is a contradiction. Thus, $Y \cap \bigcap \{A^c \colon A \in \mathcal{J}\} \neq \emptyset$. Therefore, $\langle \mathcal{F}_R, \mathcal{M} \rangle$ is an $S$-space as claimed.
\end{proof}

\begin{corollary} \label{homomorphic space}
Let $\A$ and $\mathbf{B}$ be two semilattices and $h \colon \A \rightarrow \mathbf{B}$ be an onto homomorphism. Then $\langle \mathcal{F}_{R_h}, \mathcal{M} \rangle$ is an $S$-space which is homeomorphic to $\langle X(\mathbf{B}), \mathcal{K}_{\mathbf{B}}\rangle$.
\end{corollary}

\begin{proof}
It follows from Theorem \ref{theo-F space} and an adaptation of the proof of Lemma 40 of \cite{CelaniCalomino}. 
\end{proof}

\begin{lemma}
Let $\langle X, \K \rangle$ be an $S$-space and $\mathcal{F}$ be a family of non-empty subbasic closed subsets of $\langle X, \K \rangle$. Then the family $\mathcal{M}_{\mathcal{F}}$ is a subbasis for a topology on $\mathcal{F}$.
\end{lemma}

\begin{proof}
Let $Y \in\mathcal{F}$. Then there exists $x \in X$ such that $x \in Y$. Since $\K$ is a subbase, there exists $U \in \K$ such that $x \in U$. Then $Y \in U^-_\mathcal{F}$ and $\mathcal{F} = \bigcup \{ U^-_\mathcal{F} \colon U\in \K \}$.
\end{proof}

\begin{definition}\label{lower-Vietoris-type family}
Let $\langle X, \K \rangle$ be an $S$-space. A family $\mathcal{F}$ of non-empty subbasic closed subsets of $\langle X, \K \rangle$ is a \emph{lower-Vietoris-type family} if the pair $\langle \mathcal{F},\mathcal{M}_\mathcal{F} \rangle$ is an $S$-space where $\mathcal{M}_{\mathcal{F}} = \{U^-_\mathcal{F} \colon U \in \K\}$. 
\end{definition}

We write $\mathcal{V}(X)$ for the set of lower-Vietoris-type families of $\langle X, \K \rangle$.

\begin{remark}
If $\mathcal{F}$ is a lower-Vietoris-type family of an $S$-space $\langle X, \K \rangle$, then $\mathbf{S}(\mathcal{F}) = \langle S(\mathcal{F}), \cap, \mathcal{F} \rangle$ is a semilattice where $S(\mathcal{F}) = \{ (U^-_\mathcal{F})^{c} \colon U \in \mathcal{K} \}$.
\end{remark}

Let $\A$ be a semilattice. If $\mathcal{F} \subseteq C_\K(X(\A))$ is a lower-Vietoris-type family of the dual $S$-space $\langle X(\A), \mathcal{K}_{\A} \rangle$, then we can define a relation $R_\mathcal{F} \subseteq \mathcal{F} \times X(\A)$ by
\begin{equation*}
(Y, P) \in R_\mathcal{F} \Longleftrightarrow P \in Y.
\end{equation*}

We stress that the proofs of the following three results can be performed in the same way that the proofs of Lemmas 42, 43 and Theorem 44 of \cite{CelaniCalomino}. Therefore we omit them.  

\begin{lemma}
Let $\A$ be a semilattice and $\mathcal{F} \subseteq C_\K(X(\A))$ be a lower-Vietoris-type family of the dual $S$-space of $\A$. Then $R_\mathcal{F} \subseteq \mathcal{F} \times X(\A)$ is a one-to-one meet-relation.
\end{lemma}

\begin{lemma} Let $\A$ be a semilattice and $\langle X, \K \rangle$ be an $S$-space. Then:
\begin{enumerate}
\item If $R \subseteq X \times X(\A)$ is a one-to-one meet-relation, then for every $x \in X$ and $P \in X(\A)$ we have 
\begin{equation*}
(x, P) \in R \Longleftrightarrow (R(x), P) \in R_{\mathcal{F}_R}.
\end{equation*}
\item If $\mathcal{F} \subseteq C_\K(X(\A))$ is a lower-Vietoris-type family of the dual $S$-space of $\A$, then $\mathcal{F} = \mathcal{F}_{R_\mathcal{F}}$.
\end{enumerate}
\end{lemma}

\begin{theorem}\label{homomorphic Vietoris}
Let $\A$ be a semilattice. Then the homomorphic images of $\A$ are characterized by lower-Vietoris-type families of the dual $S$-space of $\A$.
\end{theorem}

Since the class of semilattices is a variety, then by general facts from universal algebra (see II.10.9 in \cite{Sanka}) it is known that every semilattice is the homomorphic image of a freely generated free semilattice. This observation allows us to obtain the following theorem.

\begin{theorem} \label{Representation Vietoris}
Every semilattice can be represented as an $S$-space associated to a lower-Vietoris-type family.
\end{theorem}

We emphasize that it is also the case that in the class of semilattices, homomorphic images are in correspondence with congruences. In what follows, we present an explicit characterization of congruences of semilattices as lower-Vietoris-type families: let $\A$ be a semilattice and $\theta \subseteq A \times A$ be a congruence on $\A$. Let us consider the natural homomorphism $q_\theta \colon A \to A/\theta$ assigning to every $a \in A$ the equivalence class $q_\theta(a) = a/\theta$. Since $q_\theta$ is onto, by Theorem \ref{onto one-to-one} we have that $R_{q_\theta} \subseteq X(\A/\theta)\times X(\A)$ is a meet-relation one-to-one. Then
\begin{equation*}
\mathcal{F}_{R_{q_\theta}} = \{R_{q_\theta}(Q) \colon Q\in X(\A/\theta)\} \subseteq C_\K(X(\A))
\end{equation*}
is a lower-Vietoris-type family of the dual $S$-space $\langle X(\A), \mathcal{K}_{\A} \rangle$, i.e., the structure $\langle \mathcal{F}_{R_{q_\theta}}, \mathcal{M}\rangle$ is an $S$-space. 

\begin{theorem} \label{congruences semilattices}
Let $\A$ be a semilattice and $\mathcal{F} \subseteq C_\K(X(\A))$ be a lower-Vietoris-type family of the dual $S$-space of $\A$. Consider the relation
\begin{equation} \label{equa congruence F}
\theta_\mathcal{F} = \{(a,b) \in A^2 \colon [\beta(a)^{c}]^{-}_\mathcal{F} = [\beta(b)^{c}]^{-}_\mathcal{F} \}.
\end{equation}
\begin{enumerate}
\item Then $\theta_\mathcal{F}$ is a congruence on $\A$. In particular, we have that $\theta_\mathcal{F}$ is the kernel of the homomorphism $\Box_{R_\mathcal{F}} \beta\colon A \rightarrow S(\mathcal{F})$.
\item Then $Y \in \mathcal{F}$ if and only if there exists $Q \in X(\A/\theta_\mathcal{F})$ such that $Y = R_{{q_{\theta_\mathcal{F}}}}(Q)$. Moreover, $\mathcal{F}=\mathcal{F}_{R_{{q_{\theta_\mathcal{F}}}}}$.
\end{enumerate}
\end{theorem}

\begin{proof}
1. It is routine. 

2. Let $Y \in \mathcal{F}$ and let $\psi$ be the map defined in (\ref{definition of psi}). Let us consider the set $Q =q_{\theta_\mathcal{F}}[\psi(Y)]$. Then $q_{\theta_\mathcal{F}}(a) \in Q$ if and only if $a \in \psi (Y)$. It follows that $Q \in \Fi(\A/\theta_\mathcal{F})$ and since $\mathcal{F}$ is a lower-Vietoris-type family of $\langle X(\A), \K_\A \rangle$, we have that $\langle \mathcal{F}, \mathcal{M} \rangle$ is an $S$-space and satisfies condition \ref{S3}. Then it follows $Q \in X(\A/\theta_\mathcal{F})$ and $P \in R_{q_{\theta_\mathcal{F}}}(Q)$ if and only if $\psi(Y) \subseteq P$. Then $P \in Y$ and we get that $R_{q_{\theta_\mathcal{F}}}(Q)=Y$. Conversely, suppose there exists $Q \in X(\A/\theta_\mathcal{F})$ such that $Y = R_{{q_{\theta_\mathcal{F}}}}(Q)$. We consider the homeomorphism $H_\mathcal{F} \colon \mathcal{F}\to X(\mathbf{S}(\mathcal{F}))$ between $S$-spaces and let $\delta \colon \A/\theta_\mathcal{F} \to \mathbf{S}(\mathcal{F})$ be the isomorphism between semilattices given by $\delta(a/\theta_\mathcal{F}) = \{Y \in \mathcal{F} \colon Y \cap \beta(a)^c = \emptyset\}$. Therefore, $Y = H_\mathcal{F}^{-1}[\delta(Q)] \in \mathcal{F}$.
\end{proof}

\begin{theorem} \label{congruences semilattices 2}
Let $\A$ be a semilattice and $\theta$ be a congruence on $\A$. Then $\theta = \theta_{\mathcal{F}_{R_{q_{\theta}}}}$.
\end{theorem}

\begin{proof}
Let $(a,b) \in \theta$. If $Y \in [\beta(a)^{c}]^{-}_{\mathcal{F}_{R_{q_{\theta}}}}$, then there exists $Q \in X(\A/\theta)$ such that $Y=R_{q_{\theta}}(Q)$ and $R_{q_{\theta}}(Q) \cap \beta(a)^{c} \neq \emptyset$. So, there is $P \in X(\A)$ such that $(Q,P) \in R_{q_{\theta}}$ and $P \notin \beta(a)$, i.e., $q_{\theta}^{-1}[Q] \subseteq P$ and $a \notin P$. It follows that $q_{\theta}(a) \notin Q$ and since $(a,b) \in \theta$, we have $q_{\theta}(b) \notin Q$. Then $b \notin q_{\theta}^{-1}[Q]$ and by Theorem \ref{separacion} there exists $P' \in X(\A)$ such that $q_{\theta}^{-1}[Q] \subseteq P'$ and $b \notin P'$. Thus, $P' \in R_{q_{\theta}}(Q) \cap \beta(b)^{c}$ and $R_{q_{\theta}}(Q)=Y \in [\beta(b)^{c}]^{-}_{\mathcal{F}_{R_{q_{\theta}}}}$. The other inclusion is similar. Hence, $[\beta(a)^{c}]^{-}_{\mathcal{F}_{R_{q_{\theta}}}} = [\beta(b)^{c}]^{-}_{\mathcal{F}_{R_{q_{\theta}}}}$ and $(a,b) \in \theta_{\mathcal{F}_{R_{q_{\theta}}}}$. Conversely, let $(a,b) \in \theta_{\mathcal{F}_{R_{q_{\theta}}}}$. Suppose that $(a,b) \notin \theta$. Then $q_{\theta}(a) \neq q_{\theta}(b)$ and by Theorem \ref{separacion} there exists $Q \in X(\A/\theta)$ such that $a \in q_{\theta}^{-1}[Q]$ and $b \notin q_{\theta}^{-1}[Q]$. Again by Theorem \ref{separacion}, there exists $P \in X(\A)$ such that $q_{\theta}^{-1}[Q] \subseteq P$ and $b \notin P$, i.e., $P \in R_{q_{\theta}}(Q) \cap \beta(b)^{c}$ and $R_{q_{\theta}}(Q) \in [\beta(b)^{c}]^{-}_{\mathcal{F}_{R_{q_{\theta}}}}$. Then, by assumption, we have $R_{q_{\theta}}(Q) \in [\beta(a)^{c}]^{-}_{\mathcal{F}_{R_{q_{\theta}}}}$ and $R_{q_{\theta}}(Q) \cap \beta(a)^{c} \neq \emptyset$. So, there is $P' \in X(\A)$ such that $q_{\theta}^{-1}[Q] \subseteq P'$ and $a \notin P'$. On the other hand, since $a \in q_{\theta}^{-1}[Q]$ it follows that $a \in P'$, which is a contradiction. Therefore, $(a,b) \in \theta$ and $\theta = \theta_{\mathcal{F}_{R_{q_{\theta}}}}$.
\end{proof}

Observe that Theorems \ref{congruences semilattices} and \ref{congruences semilattices 2} suggest the following result.

\begin{theorem}\label{characterization congruences vietoris}
Let $\A$ be a semilattice. Then the congruences on $\A$ are characterized by lower-Vietoris-type families of the dual $S$-space of $\A$.
\end{theorem}

Something more can be said about the structure of the lower-Vietoris-type families of an $S$-space. Let $\langle X, \K \rangle$ be an $S$-space and let $\mathcal{F}_{1}, \mathcal{F}_{2}\in \mathcal{V}(X)$. Now let us regard the following relation on $\mathcal{V}(X)$:
 
\begin{displaymath}
\begin{array}{ccc}
\mathcal{F}_{1}\leq \mathcal{F}_{2} \Longleftrightarrow \forall U,V\in \mathcal{K} [ U^-_{\mathcal{F}_{2}}=V^-_{\mathcal{F}_{2}} \Rightarrow U^-_{\mathcal{F}_{1}}=V^-_{\mathcal{F}_{1}} ].
\end{array}
\end{displaymath}

\begin{theorem}\label{vietoris are a complete lattice}
Let $\langle X, \K \rangle$ be an $S$-space. Then $\langle \mathcal{V}(X), \leq \rangle$ is a complete lattice which is dually isomorphic to the set of all congruences of $\mathbf{S}(X)$. 
\end{theorem}
\begin{proof}
Let $\mathrm{Con}(\mathbf{S}(X))$ be the set of all congruences of $\mathbf{S}(X)$. We start by noticing that from Theorem \ref{characterization congruences vietoris} and Theorem 3.24 of \cite{Celani2020}, the maps $\mathrm{Con}(\mathbf{S}(X))\rightarrow \mathcal{V}(X)$, defined by $\theta\mapsto \mathcal{F}_{\theta}$ and $\mathcal{V}(X)\rightarrow \mathrm{Con}(\mathbf{S}(X))$, defined by $\mathcal{F}\mapsto \theta_{\mathcal{F}}$, are mutually inverse of each other. Then from Theorem \ref{congruences semilattices}, it is immediate that $\mathcal{F}_{1}\leq \mathcal{F}_{2}$ if and only if $\theta_{\mathcal{F}_{2}}\subseteq \theta_{\mathcal{F}_{1}}$. Hence, by general reasons $\langle \mathcal{V}(X), \leq \rangle$ is a poset which is dually isomorphic to the poset $\langle \mathrm{Con}(\mathbf{S}(X)), \subseteq \rangle$. Moreover, since $\mathrm{Con}(\mathbf{S}(X))$ is a complete lattice then if $\{\mathcal{F}_{i}\colon i\in I\}$ is a family of lower-Vietoris-type families of $\langle X, \K \rangle$, the latter allows us to define: 
\begin{equation*}
\begin{array}{ccc}
\displaystyle \bigvee \{\mathcal{F}_{i} \colon i \in I\} & = & \mathcal{F}_{\bigcap \{ \theta_{\mathcal{F}_{i}} \colon i \in I\}}
\end{array}
\end{equation*}
and
\begin{equation*}
\begin{array}{ccc}
\displaystyle \bigwedge \{\mathcal{F}_{i}\colon i\in I\} & = & \mathcal{F}_{\bigvee \{\theta_{\mathcal{F}_{i}}\colon i\in I\}}.
\end{array}
\end{equation*}
This concludes the proof.  
\end{proof}

Now we provide a characterization of the congruences of monotone semilattices by means of lower-Vietoris-type families. Unless necessary, we will omit the most of the proofs of the following results due to essentially they use the same ideas employed for the case of semilattices.   

\begin{definition}
Let $\langle X_1, \mathcal{K}_{1}, R_{1}\rangle$ and $\langle X_2, \mathcal{K}_{2}, R_{2}\rangle$ be two $mS$-spaces. A monotone meet-relation $T \subseteq X_1 \times X_2$ is \emph{one-to-one} if it is one-to-one as a meet-relation. 
\end{definition}

\begin{theorem}
Let $\langle X_1, \mathcal{K}_{1}, R_{1}\rangle$ and $\langle X_2, \mathcal{K}_{2}, R_{2}\rangle$ be two $mS$-spaces. Let $T \subseteq X_1 \times X_2$ be a monotone meet-relation. Then the monotone homomorphism $ \Box_T \colon S(X_2) \rightarrow S(X_1)$ is onto if and only if $T$ is a one-to-one monotone meet-relation.
\end{theorem}

\begin{corollary}\label{cormon}
Let $\langle \mathbf{A}, m \rangle, \langle \mathbf{B}, n \rangle \in \mathcal{MS}$ and $h \colon A \rightarrow B$ be a monotone homomorphism. Then $h$ is onto if and only if $R_{h}$ is a one-to-one monotone meet-relation. 
\end{corollary}

In the conditions of Corollary \ref{cormon}, we denote by $\lambda$ the homeomorphism of Corollary \ref{homomorphic space} between the $S$-spaces $\langle \mathcal{F}_{R_h}, \mathcal{M}\rangle$ and $\langle X(\mathbf{B}), \mathcal{K}_{\mathbf{A}} \rangle$, i.e., $\lambda \colon X(\mathbf{B}) \to \mathcal{F}_{R_h}$ is the map defined by
\begin{equation*} 
\lambda(P)=R_h(P)
\end{equation*}
for all $P\in X(\mathbf{B})$.

\begin{theorem}\label{induced homeo of mS}
Let $\langle \mathbf{A}, m \rangle, \langle \mathbf{B}, n \rangle \in \mathcal{MS}$ and $h \colon A \rightarrow B$ be an onto monotone homomorphism. Let $T \subseteq \mathcal{F}_{R_h} \times \mathcal{Z}(\mathcal{F}_{R_h})$ be the relation given by
\begin{equation*}
(R_{h}(Q), Z) \in T \Longleftrightarrow (Q, \lambda^{-1}[Z]) \in R_{n}
\end{equation*}
where $Q \in X(\mathbf{B})$ and $Z\in \mathcal{Z}(\mathcal{F}_{R_h})$. Then $\langle \mathcal{F}_{R_h}, \mathcal{M}, T \rangle$ is an $mS$-space which is homeomorphic to $\langle X(\mathbf{B}),\mathcal{K}_{\mathbf{B}}, R_{n}\rangle$ and it satisfies the following condition
\begin{equation*}
\beta(a)^c\in R_m[R_h(P)]\Longleftrightarrow H_a\in T(R_h(P))
\end{equation*}
for all $a\in A$ and $P\in X(\mathbf{B})$.
\end{theorem}

\begin{proof}
We recall that in Corollary \ref{homomorphic space} it is proved that $\mathcal{M} = \{ \lambda[V] \colon V \in \mathcal{K}_{\mathbf{B}} \}$. Moreover, we stress that $T$ is precisely the relation claimed in Corollary \ref{homeomorphisms induced by mS-spaces}. Hence, the first result follows.

Since $R_h$ is a monotone meet-relation, we have 
\begin{equation*}
\beta(a)^c\in R_m[R_h(P)] \Longleftrightarrow R_h^{-1}[\beta (a)^c]\in R_n(P)
\end{equation*}
for all $a\in A$ and $P\in X(\mathbf{B})$. It is easy to see that $\lambda[R_h^{-1}[\beta (a)^c]]=H_a$. So, by definition of $T$ we have that $R_h^{-1}[\beta (a)^c]\in R_n(P)$ if and only if $H_a\in T(R_h(P))$.
\end{proof}

In this part we study the congruences of monotone semilattices. To achieve this goal, let $\langle X, \K, R \rangle$ be an $mS$-space and let $\mathcal{F} \in \mathcal{V}(X)$. For a subset $H\subseteq \mathcal{Z}(X)$ we consider the set 
\begin{equation*}
[H\cap\mathcal{K})_{\mathcal{F}}=\{V\in \mathcal{K} \colon \exists U\in H\cap \mathcal{K}\; [U^-_\mathcal{F}\subseteq V^-_\mathcal{F}]\}.
\end{equation*}
We say that a subset $H\subseteq \mathcal{Z}(X)$ is \emph{$\mathcal{M}_\mathcal{F}$-increasing} if $[H\cap\mathcal{K})_{\mathcal{F}}=H\cap\mathcal{K}$.

\begin{definition}
Let $\langle X, \K, R \rangle$ be an $mS$-space. A family $\mathcal{F}$ of non-empty subbasic closed subsets of $\langle X, \K \rangle$ is a \emph{monotone lower-Vietoris-type family} if $\langle \mathcal{F}, \mathcal{M}_{\mathcal{F}} \rangle$ is an $S$-space and for every $Y\in \mathcal{F}$ we have that $R[Y]=\{ Z \in \mathcal{Z}(X) \colon \exists y \in Y [(y,Z) \in R] \}$ is a $\mathcal{M}_\mathcal{F}$-increasing subset.
\end{definition}

We write $\mathcal{V}_{m}(X)$ for the set of monotone lower-Vietoris-type families of an $mS$-space $\langle X, \K, R \rangle$. Note that as a consequence of Theorem \ref{induced homeo of mS}, if  $\langle \mathbf{A}, m \rangle, \langle \mathbf{B}, n \rangle \in \mathcal{MS}$ and $h \colon A \rightarrow B$ is an onto monotone homomorphism, we get that $R_m[R_h(P)]$ is $\mathcal{M}$-increasing for all $P\in X(\mathbf{B})$. So, $\mathcal{F}_{R_h}$ is a monotone lower-Vietoris-type family.

\begin{theorem} \label{theo-FM space}
Let $\langle \A, m\rangle\in \mathcal{MS}$ and $\langle X, \K,R \rangle$ be an $mS$-space. Let $T \subseteq X \times X(\A)$ be a one-to-one monotone meet-relation. Then $\mathcal{F}_T$ is a monotone lower-Vietoris-type family of $\langle X(\A), \mathcal{K}_\A,R_m\rangle$.
\end{theorem}

\begin{proof}
From Theorem 9 we have that $\langle\mathcal{F}_T,\mathcal{M}\rangle$ is an $S$-space. Let $x\in X$, we will prove that $R_m[T(x)]$ is $\mathcal{M}$-increasing. Let $a,b\in A$ be such that $H_a\subseteq H_b$ and suppose that $\beta(a)^c\in R_m[T(x)]$. We will see that $\beta(b)^c\in R_m[T(x)]$. Since $T$ is a monotone meet-relation we have $T^{-1}[\beta(a)^c]\in R(x)$. It is easy to see that $y\in T^{-1}[\beta(a)^c]$ if and only if $R(y)\in H_a$. From $H_a \subseteq H_b$ we get that $T^{-1}[\beta(a)^c]\subseteq T^{-1}[\beta(b)^c]$. So, $T^{-1}[\beta(b)^c]\in R(x)$ and this implies that $\beta(b)^c\in R_m[T(x)]$.
\end{proof}

\begin{lemma}
Let $\langle \A, m\rangle\in \mathcal{MS}$ and $\mathcal{F} \subseteq C_\K(X(\A))$ be a monotone lower-Vietoris-type family of the dual $S$-space of $\A$. Let us consider the relation $R\subseteq \mathcal{F}\times \mathcal{Z}(\mathcal{F})$ given by 
\begin{equation*}
(Y,Z)\in R \Longleftrightarrow Z\in \bigcap\{L_{H_a^c} \colon a\in A\text{ and }\beta(a)^c\notin R_m[Y]\}.
\end{equation*}
Then $\langle\mathcal{F},\mathcal{M},R\rangle$ is an $mS$-space and $R_\mathcal{F} \subseteq \mathcal{F} \times X(\A)$ is a one-to-one monotone meet-relation.
\end{lemma}

\begin{proof}
We know that $\langle\mathcal{F},\mathcal{M}\rangle$ is an $S$-space. Let $a\in A$. We will see that $m_R(H_a^c)\in S(\mathcal{F})$. Since 
\begin{equation*}
m_R(H_a^c) = \{ Y \in \mathcal{F} \colon \forall Z \in R(Y) [Z \cap H_a^c \neq \emptyset] \},
\end{equation*}
we have $Y\notin m_R(H_a^c)$ if and only if there exists $Z\in R(Y)$ such that $Z\notin L_{H_a^c}$. So, by definition of $R$ we get $\beta(a)^c\in R_m[Y]$. 

On the other hand, let $\beta(a)^c\in R_m[Y]$ and suppose that $Y\in m_R(H_a^c)$. We have $H_a\notin R(Y)$. So, there exists $b\in A$ such that $\beta(b)^c\notin R_m[Y]$ and $H_a\cap H_b^c=\emptyset$. Thus, $H_a\subseteq H_b$ and since $R_m[Y]$ is $\mathcal{M}$-increasing we get a contradiction. Therefore, $Y\notin m_R(H_a^c)$ if and only if $\beta (a)^c\in R_m[Y] $ and it is easy to see that $\beta (a)^c\in R_m[Y] $ if and only if $Y\cap m_{R_m}(\beta (a))^c\neq \emptyset$.  As $ m_{R_m}(\beta (a))^c = \beta (m a)^c$, we finally get that $m_R(H_a^c)^c=H_{ma}$ and then $m_R(H_a^c)\in S(\mathcal{F})$. It is immediate to see that $R(Y) = \bigcap\{L_{H_a^c} \colon a \in A \text{ and }Y\in m_R(H_a^c)\}$. So, $\langle\mathcal{F},\mathcal{M},R\rangle$ is an $mS$-space.

We will see now that $R_\mathcal{F}$ is a monotone one-to-one relation. Let $a\in A$. We only need to prove that 
\begin{equation*}
\beta(a)^c\in R_m[Y]\Longleftrightarrow R_\mathcal{F}^{-1}[\beta (a)^c]\in R(Y).
\end{equation*}
Note that $R_\mathcal{F}^{-1}[\beta (a)^c]=H_a$. Then 
\begin{equation*}
\beta(a)^c\in R_m[Y]\Longleftrightarrow Y\notin m_R(H_a^c) \Longleftrightarrow  H_a\in R(Y)
\end{equation*}
and the assertion follows.
\end{proof}

\begin{theorem}\label{homomorphic Vietoris mon}
Let $\langle \A,m\rangle\in\mathcal{MS}$. Then the homomorphic images of $\langle \A,m\rangle$ are characterized by monotone lower-Vietoris-type families of the dual $S$-space of $\A$.
\end{theorem}

As we saw on Remark \ref{Monotonic semilattices are a variety}, the class of monotone semilattices is a variety so we obtain a representation theorem analogue to Theorem \ref{Representation Vietoris}.

\begin{theorem}
Every monotone semilattice can be represented as an $mS$-space associated to a monotone lower-Vietoris-type family.
\end{theorem}

We conclude this section by providing a characterization of the congruences on monotone semilattices. The proofs of the following two results are similar to the proofs of Theorem \ref{congruences semilattices} and \ref{congruences semilattices 2}, respectively, so we omit them.

\begin{theorem}
Let $\langle \mathbf{A}, m \rangle \in \mathcal{MS}$.
\begin{enumerate}
\item If $\mathcal{F} \subseteq C_\K(X(\A))$ is a monotone lower-Vietoris-type family of the dual $mS$-space of $\langle \mathbf{A}, m \rangle$, then $\theta_{\mathcal{F}}$ given by (\ref{equa congruence F}) is a congruence on $\langle \mathbf{A}, m \rangle$ such that $\mathcal{F}=\mathcal{F}_{R_{{q_{\theta_\mathcal{F}}}}}$.
\item If $\theta$ is a congruence on $\langle \mathbf{A}, m \rangle$, then $\theta = \theta_{\mathcal{F}_{R_{q_{\theta}}}}$.
\end{enumerate}
\end{theorem}

\begin{theorem}
Let $\langle \mathbf{A}, m \rangle \in \mathcal{MS}$. Then the congruences on $\langle \mathbf{A}, m \rangle$ are characterized by monotone lower-Vietoris-type families of the dual $mS$-space of $\langle \mathbf{A}, m \rangle$.
\end{theorem}

Let $\langle X, \K, R \rangle$ be an $mS$-space. 
If we endow $\mathcal{V}_{m}(X)$ with the induced order on $\mathcal{V}(X)$ then we obtain an analogue of Theorem \ref{vietoris are a complete lattice} for $mS$-spaces.

\begin{theorem}
For every $mS$-space $\langle X, \K, R \rangle$, the poset $\langle\mathcal{V}_{m}(X), \leq\rangle$ is a complete lattice which is dually isomorphic to the set of all congruences of $\langle \mathbf{S}(X), m_{R}\rangle$. 
\end{theorem}

\subsection{Congruence lattices of semilattices and lattices of algebraic sets}

In \cite{Fajtlowicz}, Fajtlowicz and Schmidt developed a duality between congruences of semilattices and the lattices of algebraic sets. Unlike us, they worked with ideals instead of filters. By following the same line and applying our characterization of congruences of semilattices, we conclude this paper by showing that there exists a duality between congruences of semilattices and families of filters.

A subset $X$ of a complete lattice $\mathbf{L}$ is an \emph{algebraic subset} if it contains $1$, it is closed under arbitrary meets, and it is closed under non-empty directed joins. Following the notation of \cite{Hyndman}, we denote ${\rm{S}}_p({\mathbf{L}})$ the lattice of all algebraic subsets of $\mathbf{L}$.

\begin{definition}
Let $\A$ be a semilattice and $\theta$ be a congruence on $A$. A filter $F$ of $\A$ is \emph{$\theta$-closed} if $(x,y) \in\theta$ and $y \in F$ implies $x \in F$.
\end{definition}

Let $\A$ be a semilattice and let $\{F_{i}\colon i\in I\}$ be a family of filters of $\A$. We recall that $\Fi(\A)$ is a complete lattice whose operations are given by 
\begin{equation*}
\begin{array}{ccc}
\displaystyle \bigvee \{ F_{i} \colon i \in I \} = F \left( \bigcup \{ F_{i} \colon i \in I \} \right) & \text{ and } & \displaystyle \bigwedge \{ F_{i} \colon i \in I \} = \bigcap \{ F_{i} \colon i \in I \}.
\end{array}
\end{equation*}

The following Theorem can be found in \cite{KIRA ADARICHEVA and J. B. NATION}, using ideals instead of filters. We include the proof for completeness.
\begin{theorem}
Let $\A$ be a semilattice. Then the lattice of congruences of $\A$ is dually isomorphic to ${\rm{S}}_p(\Fi(\A))$.
\end{theorem}

\begin{proof}
We write ${\rm{Con}}(\A)$ the lattice of all congruences of $\A$. We define the map $\sigma \colon {\rm{Con}}(\A) \rightarrow {\rm{S}}_p(\Fi(\A))$ by setting 
$\sigma(\theta)$ as the set of all $\theta$-closed filters of $\A$. Now, let us consider the map $\rho \colon {\rm{S}}_p(\Fi(\A)) \rightarrow {\rm{Con}}(\A)$ defined by 
\begin{equation*}
\rho(U) = \{(x, y) \in A^2 \colon \forall F \in U (x \in F \Longleftrightarrow y \in F)\}.
\end{equation*}
It is straightforward to check that $\sigma$ and $\rho$ are order-reversing, $\sigma(\theta) \in {\rm{S}}_p(\Fi(\A))$ and $\rho(U) \in {\rm{Con}}(\A)$. In order to show that $\theta = \rho \sigma (\theta)$, note that if $y < x$ and $(x, y) \notin \theta$, then $\{z \in A \colon (x \wedge z, x) \in \theta \}$ is a $\theta$-closed filter containing $x$ and but not $y$. Hence $(x,y) \notin \rho \sigma(\theta)$. To see that $U = \sigma \rho (U)$, consider a filter $F \notin U$. For each $x \in A$, let $F_x = \bigcap \{F \in U \colon x \in F\}$. So, $F_x \in U$. Then $\{F_x \colon x \in F\}$ is directed, whence $\bigcup \{F_x \colon x \in F \} \in U$. Thus, the latter union properly contains $F$. Then there exist $y < x$ with $x \in F$, $y \in F_x$ and $y \notin F$, i.e., $F$ is not $\rho (U)$-closed. Therefore, $F \notin U$ implies $F \notin \sigma\rho(U)$. This concludes the proof.
\end{proof}

Let $\A$ be a semilattice and $\theta$ be a congruence on $\A$. Then for every $Q \in X(\A/\theta)$ we have $R_{q_\theta}(Q) \in C_\K(X(\A))$. So, there is an associated filter $\psi(R_{q_\theta}(Q))=q_\theta^{-1}[Q]$ and we get a family of filters
\begin{equation*}
\Psi (\mathcal{F}_{R_{q_\theta}}) = \{ q_\theta^{-1}[Q] \colon Q\in X(\A/\theta)\}.
\end{equation*}

\begin{lemma}
Let $\A$ be a semilattice and $\theta$ be a congruence on $\A$. Then:
\begin{enumerate}
\item $\Psi (\mathcal{F}_{R_{q_\theta}})\subseteq \sigma (\theta)$.
\item If $F\in \sigma(\theta) - \{ A \}$, then 
\begin{equation*}
F = \displaystyle \bigcap \{ q_{\theta}^{-1}[Q] \in \Psi(\mathcal{F}_{R_{q_\theta}}) \colon F \subseteq q_{\theta}^{-1}[Q] \}.
\end{equation*}
\end{enumerate}
\end{lemma}

\begin{proof}
1. It is easy to see that $q_\theta^{-1}[Q]$ is a $\theta$-closed filter for all $Q \in X(\A/\theta)$.

2. Let $F\in \sigma(\theta) - \{ A \}$. Then $q_\theta [F] \in \Fi(\A/\theta)$ and as a consequence of Theorem \ref{separacion} we get $q_\theta [F] =\bigcap \{ Q \in X(\A/\theta) \colon q_{\theta} [F] \subseteq Q\}$. It follows that $F = \bigcap \{ q_{\theta}^{-1}[Q] \in \Psi(\mathcal{F}_{R_{q_\theta}}) \colon F\subseteq q_\theta^{-1}[Q] \}$.
\end{proof}


%
%


\begin{thebibliography}{99}

\bibitem{KIRA ADARICHEVA and J. B. NATION} Adaricheva, K., Nation, J.: Lattices of quasi-equational theories as congruence lattices of semilattices with operators: part I. Int. J. Algebra Comput. {\bf 22}(7), 27 p. (2012)

\bibitem {Guram-Jansana I} Bezhanishvili, G., Jansana, R.: Priestley style duality for distributive meet-semilattices. Stud. Log. {\bf{98}}(1-2), 83--122 (2011)

\bibitem {Guram-Jansana II} Bezhanishvili, G., Jansana, R.: Generalized Priestley quasi-orders. Order {\bf{28}}(2), 201--220 (2011)

\bibitem {Guram-Jansana III} Bezhanishvili, G., Jansana, R.: Esakia style duality for implicative semilattices. Appl. Categ. Struct. {\bf{21}}(2), 181--208 (2013)

\bibitem{Birkhoff} Birkhoff, G.: Rings of sets. Duke Math. J. {\bf{3}}, 443--454 (1937)

\bibitem{Sanka} Burris, S., Sankappanavar, H.P.: A Course in Universal algebra. Springer Verlag (1981).

\bibitem{CelaniTopological} Celani, S.: Topological representation of distributive semilattices. Sci. Math. Jpn. {\bf{58}}(1), 55--66 (2003)  

\bibitem{CelaniCalomino} Celani, S., Calomino, I.: Some remarks on distributive semilattices. Commentat. Math. Univ. Carol. {\bf{54}}(3), 407--428 (2013)

\bibitem{Celani2020} Celani, S., Gonz\'{a}lez, L.J.: A categorical duality for semilattices and lattices. Appl. Categ. Struct. {\bf{28}}(5), 853--875 (2020)

\bibitem{menchon2} Celani, S., Mench\'{o}n, M.P.: Monotonic distributive semilattices. Order {\bf{36}}(3), 463--486 (2019)

\bibitem{Chellas} Chellas, B.: Modal Logic. An introduction. Cambridge University Press (1980).

\bibitem{PalmigianoDunn} Dunn, J., Gehrke, M., Palmigiano, A.: Canonical extensions and relational completeness of some substructural logics. J. Symb. Log. {\bf{70}}(3), 713--740 (2005) 

\bibitem{Duntsch-Orlowska-RewitzkyMultirelations2010} D\"{u}ntsch, I., Orlowska, E., Rewitzky, I.: Structures with multirelations, their discrete dualities and applications. Fundam. Inform. {\bf{100}}(1-4), 77--98 (2010) 

\bibitem{Fajtlowicz} Fajtlowicz, S., Schmidt J.: Bezout families, join-congruences, and meet-irreducible ideals. Lattice Theory, Szeged, 1974, Colloq. Math. Soc. Janos Bolyai 14, 51--76 (1976)

\bibitem{Gehrke -Jonsson2000} Gehrke, M., J\'{o}nsson, B.: Monotone bounded distributive lattice expansions. Math. Jap. {\bf{52}}(2), 197--213 (2000)

\bibitem{GJ2} Gehrke, M., J\'{o}nsson, B.: Bounded distributive lattice expansions. Math. Scand. {\bf{94}}(1), 13--45 (2004)

\bibitem{GPsemilattices} Gouveia, M.J., Priestley, H.A.: Canonical extensions and profinite completions of semilattices and lattices. Order {\bf{31}}(2), 189--216 (2014)

\bibitem{Gr78} Gr{ä}tzer, G.: General Lattice Theory. Birkh\"{a}user Verlag (1998)

\bibitem{Hansen} Hansen, H.H.: Monotonic Modal Logic. Master's thesis, University of Amsterdam (2003)

\bibitem{Hyndman} Hyndman, J., Nation, J.B., Nishida, J.: Congruence lattices of semilattices with operators. Stud. Log. {\bf{104}}(2), 305--316 (2016)

\bibitem{Ivanova-dimova} Ivanova-Dimova, E.: Lower-Vietoris-type topologies on hyperspaces. Topology Appl. {\bf{220}}, 100--110 (2017)

\bibitem{Jonsson y Tarski} J\'{o}nsson, B., Tarski, A.: Boolean algebras with operators. I. Am. J. Math. {\bf{73}}, 891--939 (1951)

\bibitem{JT2} J\'{o}nsson, B., Tarski, A.: Boolean algebras with operators. II. Am. J. Math. {\bf{74}}, 127--162 (1952)

\bibitem{MoshierJipsen} Moshier, M.A., Jipsen, P.: Topological duality and lattice expansions. I: A topological construction of canonical extensions. Algebra Univers. {\bf{71}}(2), 109--126 (2014)

\bibitem {Priestley I} Priestley H.: Representation of distributive lattices by means of ordered Stone spaces. Bull. London Math. Soc. {\bf{2}}, 186--190 (1970)

\bibitem{Rewitzky2003} Rewitzky, I.: Binary multirelations. Theory and applications of relational structures as knowledge instruments. Lect. Notes Comput. Sci. {\bf{2929}}, 256--271 (2003) 

\bibitem{Stone2} Stone M.: Topological representations of distributive lattices and Brouwerian logics. \v{C}as. Mat. Fys. {\bf{67}}, 1--25 (1937)

\end{thebibliography}
\end{document}